\documentclass[twoside,british,english]{svjour3}
\usepackage[T1]{fontenc}
\usepackage{float}
\usepackage{units}
\usepackage{url}
\usepackage{amsmath}
\usepackage{amssymb}
\usepackage{graphicx}
\usepackage[authoryear]{natbib}

\makeatletter
\numberwithin{equation}{section}
\numberwithin{figure}{section}
  \newenvironment{svmultproof}{\begin{proof}}{\qed\end{proof}}

\AtBeginDocument{

}

\makeatother

\usepackage{babel}
  \addto\captionsbritish{}
  \addto\captionsbritish{}
  \addto\captionsbritish{}
  \addto\captionsbritish{}
  \addto\captionsbritish{}
  \addto\captionsbritish{}
  \addto\captionsbritish{}
  \addto\captionsenglish{}
  \addto\captionsenglish{}
  \addto\captionsenglish{}
  \addto\captionsenglish{}
  \addto\captionsenglish{}
  \addto\captionsenglish{}
  \addto\captionsenglish{}

\begin{document}

\title{Approximation and estimation of very small probabilities of multivariate
extreme events}

\titlerunning{Approximation/estimation of very small probabilities of multivariate
extreme events}

\author{Cees de Valk}

\institute{CentER, Tilburg University, P.O. Box 90153, 5000 LE Tilburg, The
Netherlands. Email: \foreignlanguage{british}{C}.F.deValk@uvt.nl/ceesfdevalk@gmail.com}

\date{December 16, 2015}
\maketitle
\begin{abstract}
This article discusses modelling of the tail of a multivariate distribution
function by means of a large deviation principle (LDP), and its application
to the estimation of the probability $p_{n}$ of a multivariate extreme
event from a sample of $n$ \emph{iid} random vectors, with $p_{n}\in[n^{-\tau_{2}},n^{-\tau_{1}}]$
for some $\tau_{1}>1$ and $\tau_{2}>\tau_{1}$. One way to view classical
tail limits is as limits of probability ratios. In contrast, the tail
LDP provides asymptotic bounds or limits for log-probability ratios.
After \foreignlanguage{british}{standardising} the marginals to standard
exponential, dependence is represented by a homogeneous rate function
$I$. Furthermore, the tail LDP can be extended to represent both
dependence and marginals, the latter implying marginal log-GW tail
limits. A connection is established between the tail LDP and residual
tail dependence (or hidden regular variation) and a recent extension
of it. Under a smoothness assumption, they are implied by the tail
LDP. Based on the tail LDP, a simple estimator for very small probabilities
of extreme events is formulated. It avoids estimation of $I$ by making
use of its homogeneity. Strong consistency in the sense of convergence
of log-probability ratios is proven. Simulations and an application
illustrate the difference between the classical approach and the LDP-based
approach.\end{abstract}

\begin{quote}
\textbf{Mathematics Subject Classification (20}1\textbf{0)}: 60F10,
60G70, 62G32
\end{quote}

\section{Introduction\label{sec:Introduction}}

In this article, we will consider estimation of very small probabilities
$p_{n}$ of multivariate extreme events from a sample of size $n$,
with
\begin{equation}
p_{n}\in[n^{-\tau_{2}},n^{-\tau_{1}}]\quad with\quad\tau_{2}>\tau_{1}>1,\label{eq:p_n}
\end{equation}
motivated by applications requiring quantile estimates for $p_{n}\ll1/n$
in \emph{e.g.} flood protection and more generally, natural hazard
assessment, and in operational risk assessment for financial institutions.
Multivariate events with such low probabilities are also relevant
to these fields of application. Examples are breaching of a flood
protection consisting of multiple sections differing in exposure,
design and maintenance along a shoreline or river bank (\citet{Steenbergen}),
damage to an offshore structure caused by the combined effects of
multiple environmental loads like water level, wave height, etc. (\citet{ISO}),
and operational losses suffered by banks in different business lines
and due to various types of events (\citet{Embrechts}). 

Most work on estimation of probabilities of extreme events is based
on the regularity assumption that the distribution function $F$ is
in the domain of attraction of some extreme value distribution function
(\citet{Laurens  boek,Resnick book}). In the univariate case, this
is equivalent to the generalised Pareto (GP) tail limit
\begin{equation}
\lim_{t\rightarrow\infty}t(1-F(xw(t)+U(t))=1/h_{\gamma}^{-1}(x)\quad\forall x\in h_{\gamma}((0,\infty))\label{eq:GP}
\end{equation}
for some positive function $w$, with $U(t):=F^{-1}(1-1/t)$ and
\begin{equation}
h_{\gamma}(\lambda):=\begin{cases}
(\lambda^{\gamma}-1)/\gamma & \textrm{if}\:\gamma\neq0\\
\log\lambda & \textrm{if}\:\gamma=0
\end{cases}\label{eq:h}
\end{equation}
for some $\gamma\in\mathbb{R}$, the extreme value index. In the multivariate
case, with $F$ the distribution function of a random vector $X=(X_{1},..,X_{m})$
with continuous marginals $F_{1},..,F_{m}$, it implies that each
marginal satisfies the GP tail limit (\ref{eq:GP}) and that $V:=(V_{1},..,V_{m})$,
the random vector with standard Pareto marginals with

\begin{equation}
V_{j}:=(1-F_{j}(X_{j}))^{-1}\label{eq:Vdef}
\end{equation}
for $i=1,..,m$, satisfies
\begin{equation}
\lim_{t\rightarrow\infty}tP\left(V\in tA\right)=\nu(A)\label{eq:conv_expmeasure}
\end{equation}
for every Borel set $A\subset[0,\infty)^{m}$ such that $\inf_{x\in A}\max(x_{1},..,x_{m})>0$
and $\nu(\partial A)=0$, with $\nu$ a measure satisfying $\nu(Aa)=a^{-1}\nu(A)$
for all these $A$ and all $a>0$. Based on the GP tail limit and
the exponent measure $\nu$ or its properties, estimators for probabilities
have been formulated; \emph{e.g.} \citet{Smithetal}, \citet{ColesTawn1990,Coles=000026Tawn},
\citet{Joe}, \citet{BruunTawn}, \citet{deHaan=000026Sinha}, \citet{Drees =000026 de Haan}. 

If the maxima of some components of $X$ under consideration are asymptotically
independent, these estimators may produce invalid results. To alleviate
this problem, residual tail dependence (RTD), also known as hidden
regular variation, was introduced as an additional regularity assumption
on the tail of the multivariate survival function $F^{c}$, defined
by $F^{c}(x)=P(X_{i}>x_{i}\:\forall i\in\{1,..,m\})$; \emph{e.g.}
\citet{Ledford=000026Tawn96,Ledford=000026Tawn97,Ledford=000026Tawn98},
\citet{LiPeng}, \citet{Resnick2002}, \citet{Draisma} and \citet{HeffResnick2005}.
This model was recently extended in \citet{Wadsworth=000026Tawn}.
Another approach, based on conditional limits, was proposed in \citet{HeffernanTawn}
and \citet{HeffernanResnick}.

The first-order tail regularity conditions (\ref{eq:GP}) and (\ref{eq:conv_expmeasure})
can be seen as limiting relations for probability ratios, so they
only allow estimation of probabilities $p_{n}$ vanishing slowly enough,
that is,
\begin{equation}
p_{n}\geq\lambda k_{n}/n\label{eq:pn_for_GP}
\end{equation}
for some $\lambda>0$ and some intermediate sequence $(k_{n})$, and
therefore, $np_{n}\rightarrow\infty$ as $n\rightarrow\infty$. For
an \emph{iid} sample, the empirical probability $\hat{p}_{n}$ is
an unbiased estimator for such $p_{n}$, satisfying that $\hat{p}_{n}/p_{n}\overset{p}{\rightarrow}1$
(from the binomial distribution of $n\hat{p}_{n}$). Therefore, estimators
for these $p_{n}$ which make use of tail regularity can at best achieve
a reduction in variance when compared to $\hat{p}_{n}$. To allow
tail extrapolation to be carried further to more rapidly vanishing
$p_{n}$, additional assumptions beyond (\ref{eq:GP}) and (\ref{eq:conv_expmeasure})
are introduced. Initially, \emph{e.g.} in \citet{Smithetal}, \citet{ColesTawn1990,Coles=000026Tawn}
and \citet{Joe}, the tail is assumed to follow the limiting distribution
exactly above some thresholds, so likelihood methods can be employed.
Later, \emph{e.g.} in \citet{deHaan=000026Sinha}, \citet{LiPeng},
\citet{Drees =000026 de Haan}, \citet{Draisma} and \citet{Laurens  boek},
convergence to the limiting distribution and its effect on bias in
estimates is explicitly considered. For the marginals, additional
assumptions on convergence to the limit (\ref{eq:GP}) in these articles
are identical to or stronger than those invoked for univariate quantile
estimation%
\footnote{Common assumptions are strong second-order extended regular variation
as in \emph{e.g.} Theorem 4.3.1(1) of \citet{Laurens  boek} or the
Hall class (\citet{Hall}).%
}. However, the latter appear to be restrictive when $\gamma=0$, regardless
of the precise nature of the assumption; see \citet{De Valk I}, Proposition
1. For example, they exclude the normal and the lognormal distribution,
but also for all $\alpha\in(0,\infty)\setminus\{1\}$ the distribution
functions of $Y^{\alpha}$ with $Y$ exponentially distributed, and
of $\exp((\log V)^{\alpha})$ with $V$ Pareto distributed. 

To overcome these limitations, we will consider a different approach
in this paper. Rather than imposing additional assumptions on convergence
beyond the first-order limits (\ref{eq:GP}) and (\ref{eq:conv_expmeasure}),
we will attempt to replace them by different types of first-order
limits more suitable for the probability range (\ref{eq:p_n}). Suppose
that $(k_{n})$ satisfies $k_{n}\leq n^{c}$ for some $c\in(0,1)$.
Then $(p_{n})$ satisfying (\ref{eq:p_n}) does not satisfy (\ref{eq:pn_for_GP}),
but
\begin{equation}
\tau_{1}\leq\frac{\log p_{n}}{\log(k_{n}/n)}\leq\tau_{2}/(1-c)<\infty.\label{eq:log_pn}
\end{equation}

This suggests that replacing the classical limits of probability ratios
by limits of log-probability ratios could provide a framework for
constructing estimators for probabilities of extreme events in the
range (\ref{eq:p_n}). 

In the next section, we address the limiting behaviour of log-probability
ratios in the univariate case as introduction to the multivariate
case. We will find that this behaviour is described by a large deviation
principle (LDP) (see \emph{e.g.} \citet{Dembo=000026Zeitouni}). It
is generalised to the multivariate setting in Section \ref{sec:Dependence}.
In Section \ref{sec:residual}, we establish a connection between
the tail LDP and residual tail dependence and related assumptions.
Section \ref{sec:estimator} returns to the basic LDP and applies
it to formulate a simple estimator for probabilities of extreme events
in the range (\ref{eq:p_n}) and to prove its consistency. In Section
\ref{sec:Examples}, this estimator is compared to its classical analogues
in simulations, and an application of the LDP-based estimator is presented
as illustration. Section \ref{sec:Discussion} closes with a discussion
of the results and of outstanding issues. Readers primarily interested
in tail dependence could scan Section 2 for the approach and background,
read the first part of Section 3 until eq. (\ref{eq:LDP_mexp_Q}),
and then continue with Sections 4-7. Lemmas can be found in Section
\ref{sec:lemmas}. 

The following notation is adopted: Id denotes the identity. The interior
of a set $S$ is denoted by $S^{o}$ and its closure by $\bar{S}$.
The image of a set $S$ under a function $f$ is written as $f(S)$.
The infimum of an (extended) real function $f$ over $S$ is written
as $\inf f(S)$; by convention, $\inf\{\emptyset\}:=\infty$. To avoid
tedious repetition, expressions of the form $a\leq\liminf_{y\rightarrow\infty}f(y)\leq\limsup_{y\rightarrow\infty}f(y)\leq b$
are abbreviated to $a\leq\liminf_{y\rightarrow\infty}f(y)\leq\limsup_{y\rightarrow\infty}...\leq b$.

\section{\label{sec:univar}Introducing the tail LDP: the univariate case}

We will begin by examining the univariate case in order to become
acquainted with a particular type of large deviation principle (LDP)
as a model of the tail of a distribution function. 

Let $X$ be a real-valued random variable and let $\{b_{y},\: y>0\}$
be a family of real functions such that for $D\subset[0,\infty)$,
$b_{y}(D)$ becomes more extreme in some sense when $y$ is increased.
In line with the classical limits (\emph{e.g.} (\ref{eq:GP})), we
could consider an affine function for $b_{y}$,\emph{ i.e.}, $b_{y}(x)=r(y)+g(y)x,$
with $r$ some nondecreasing function and $g$ some measurable positive
function. Instead, for a reason to be explained later, we will assume
that $F(0)<1$ and consider
\begin{equation}
b_{y}(x)=r(y)\textrm{e}^{g(y)x}\label{eq:By-logX}
\end{equation}
with $g$ and $r$ as above and $r(\infty)>0$. We will examine the
limiting behaviour of
\begin{equation}
\frac{1}{y}\log P(X\in b_{y}(D))\label{eq:logP}
\end{equation}
as $y\rightarrow\infty$. Substituting $y_{n}=-\log(k_{n}/n)$ for
$y$, this determines the behaviour of the log-probability ratio in
(\ref{eq:log_pn}) with $p_{n}=P(X\in b_{y_{n}}(D))$ as $n\rightarrow\infty$.

Generally speaking, normalised logarithms of probabilities like (\ref{eq:logP})
do not need to satisfy limits, so we will only assume that%
\footnote{See the end of Section \ref{sec:Introduction} for the notation employed
here.%
}
\begin{equation}
J(D^{o})\leq\liminf_{y\rightarrow\infty}\frac{1}{y}\log P(X\in b_{y}(D))\leq\limsup_{y\rightarrow\infty}...\leq J(\bar{D})\label{eq:LDP_general}
\end{equation}
for (at least) $D=(x,\infty)$ for all $x\geq0$, with $J$ some monotonic
set function taking values in $[0,\infty]$. Noting that $\varphi(x):=-J((x,\infty))$
is nondecreasing in $x$, we have at every continuity point $x$ of
$\varphi$ in $(0,\infty)$,
\begin{equation}
\lim_{y\rightarrow\infty}\frac{1}{y}\log(1-F(\textrm{e}^{g(y)x}r(y)))=-\varphi(x).\label{eq:a_limit}
\end{equation}

Let $q$ be the left-continuous inverse of $-\log(1-F)$, so
\begin{equation}
q:=F^{-1}(1-\textrm{e}^{-\textrm{Id}})=U\circ\exp.\label{eq:defq}
\end{equation}

Assume that $\varphi$ is not constant. By Lemma 1.1.1 of \citet{Laurens  boek},
(\ref{eq:a_limit}) implies $\lim_{y\rightarrow\infty}(\log q(y\lambda)-\log r(y))/g(y)=\varphi^{-1}(\lambda)$
at every continuity point of the left-continuous inverse $\varphi^{-1}$
of $\varphi$ in $(\varphi(0),\varphi(\infty))$. Therefore (\emph{cf.}
the proof of Theorem 1.1.3 in \citet{Laurens  boek}), we may take
$r=q$ and choose $g$ measurable and such that $\varphi^{-1}(\lambda)=h_{\theta}(\lambda)$
for some real $\theta$ (see (\ref{eq:h})). As a result,
\begin{equation}
\lim_{y\rightarrow\infty}\frac{\log q(y\lambda)-\log q(y)}{g(y)}=h_{\theta}(\lambda)\quad\forall\lambda>0\label{eq:logq-ERV}
\end{equation}
and from (\ref{eq:a_limit}),
\begin{equation}
\lim_{y\rightarrow\infty}\frac{1}{y}\log(1-F(\textrm{e}^{g(y)x}q(y)))=-h_{\theta}^{-1}(x)\quad\forall x\in h_{\theta}((0,\infty)).\label{eq:logGW-in-F}
\end{equation}

Eq. (\ref{eq:logq-ERV}) states that $\log q$ is extended regularly
varying with index $\theta$. By (\ref{eq:logq-ERV}), $\lim_{y\rightarrow\infty}g(y\lambda)/g(y)=\lambda^{\theta}$
for all $\lambda>0$, so $g\in RV_{\theta}$ ($g$ is regularly varying
with index $\theta$); see Appendix B of \citet{Laurens  boek}. Now
(\ref{eq:LDP_general}) can be fully specified:
\begin{proposition}
(a) \label{pro:PropA}Suppose that asymptotic bounds (\ref{eq:LDP_general})
with (\ref{eq:By-logX}) apply to all $D$ of the form $D=(x,\infty)$
with $x\geq0$, with $J$ monotonic and $x\mapsto J((x,\infty))$
non-constant. Then $g$ in (\ref{eq:By-logX}) can be chosen such
that (\ref{eq:LDP_general}) holds with $r=q$ and $J=-\inf h_{\theta}^{-1}(\textrm{Id})$
for some $\theta\in\mathbb{R}$ for every Borel set $D\subset[0,\infty)$,
i.e.,
\[
-\inf h_{\theta}^{-1}(D^{o})\leq\liminf_{y\rightarrow\infty}\frac{1}{y}\log P\left(\frac{\log X-\log q(y)}{g(y)}\in D\right)
\]
\begin{equation}
\leq\limsup_{y\rightarrow\infty}...\leq-\inf h_{\theta}^{-1}(\bar{D});\label{eq:LDP-log-GW}
\end{equation}

(b) Eq. (\ref{eq:LDP-log-GW}) is equivalent to (\ref{eq:logq-ERV}),
which is equivalent to (\ref{eq:logGW-in-F}).\end{proposition}
\begin{svmultproof}
We have proven that (\ref{eq:LDP_general}) for $D=(x,\infty)$ implies
the equivalent limit relations (\ref{eq:logq-ERV}) and (\ref{eq:logGW-in-F}),
so it remains to be shown that (\ref{eq:logGW-in-F}) implies (\ref{eq:LDP-log-GW})
for every Borel set $D\subset[0,\infty)$. The lower bound holds if
$D^{o}$ is empty. Else, with $\alpha:=\inf h_{\theta}^{-1}(D^{o})\geq0$
and $\delta>0$ such that $(h_{\theta}(\alpha),h_{\theta}(\alpha+\delta)]\subset D^{o}$
and for every $\varepsilon\in(0,\delta/2)$, $P((\log X-\log q(y))/g(y)\in D^{o})$
$\geq F(\textrm{e}^{g(y)h_{\theta}(\alpha+\delta)}q(y))-F(\textrm{e}^{g(y)h_{\theta}(\alpha)}q(y))$
$\geq\textrm{e}^{-y(\alpha+\varepsilon)}-\textrm{e}^{-y(\alpha+\delta-\varepsilon)}$
$\geq\textrm{e}^{-y(\alpha+\varepsilon)}(0\vee1-\textrm{e}^{-y(\delta-2\varepsilon)})$,
provided that $y$ is large enough, as a consequence of (\ref{eq:logGW-in-F}).
As $\delta>0$ is arbitrary, this implies the lower bound in (\ref{eq:LDP-log-GW}).
The proof of the upper bound is similar and is therefore omitted. 
\end{svmultproof}

The pair of equivalent limit relations (\ref{eq:logq-ERV}) and (\ref{eq:logGW-in-F})
was named the log-Generalised Weibull (log-GW) tail limit in \citet{De Valk I},
where it was proposed as a model for estimating high quantiles for
probabilities in the range (\ref{eq:p_n}), as an alternative to the
more familiar GP tail limit. If $\theta=0$ and $g(y)\rightarrow g_{\infty}>0$
as $y\rightarrow\infty$, it reduces to the Weibull tail limit; see
\emph{e.g.} \citet{Broniatowski} and \citet{Kl=0000FCppelberg}. 

The log-GW tail limit looks deceptively similar to a GP tail limit,
but it is a very different beast, primarily due to the logarithm in
(\ref{eq:logGW-in-F}) (or equivalently, due to the exponent in (\ref{eq:defq})).
Its domain of attraction covers a wide range of tail weights: a class
of light tails having finite endpoints, tails with Weibull limits
(such as the normal distribution), all tails with classical Pareto
tail limits and, more generally, with log-Weibull tail limits. For
the latter, $F\circ\exp$ satisfies a Weibull tail limit; an example
is the lognormal distribution. For estimation of high quantiles with
probabilities (\ref{eq:p_n}) of distribution functions within the
domain of attraction of the GP limit with $\gamma=0$, the log-GW
tail limit offers a continuum of limits instead of just one; as a
consequence, it is much more widely applicable (see \citet{De Valk I}).
Readers more comfortable with classical tail limits may consider focusing
on tails with a Pareto tail limit ($\gamma>0)$, which have a log-GW
limit with $\theta=1$ (so $h_{\theta}(\lambda)=\lambda-1$) and $g(y)=\gamma y$.
This may make reading of the rest of the article easier. 

An expression of the form (\ref{eq:LDP-log-GW}) is an example of
a large deviation principle%
\footnote{An LDP on a topological space $\mathcal{T}$ is an expression of the
form (\ref{eq:LDP-log-GW}) with

$P((\log X-\log q(y))/g(y)\in D)$ generalised to $\mu_{y}(D)$, with
$\{\mu_{y},\: y>0\}$ some family of probability measures on the Borel
$\sigma$-algebra, and $h_{\theta}^{-1}$ generalised to some rate
function (= lower semicontinuous function) $I$; the expression is
supposed to hold for every Borel set $D$ in $\mathcal{T}$.%
} (LDP); see Section 1.2 of \citet{Dembo=000026Zeitouni} for a general
background. The \emph{rate function} of the LDP (\ref{eq:LDP-log-GW})
is $h_{\theta}^{-1}$. The bounds provided by an LDP are crude; for
example, they are unaffected by multiplying the probability in (\ref{eq:LDP-log-GW})
by a positive number. One could see this as the price to be paid for
approximating probabilities over a very wide range. More precise bounds
may exist, but such cases should be regarded as the exception rather
than the rule. Observe also that the bounds do not involve integration
and in fact, most of $D$ does not even matter to the values of the
bounds. The LDP (\ref{eq:LDP-log-GW}) reduces to a limit only if
$D$ satisfies $\inf h_{\theta}^{-1}(D^{o})=\inf h_{\theta}^{-1}(\bar{D})$;
such a $D$ is called a \emph{continuity set} of the rate function.

Had we considered events of the form $b_{y}(x)=r(y)+g(y)x$ instead
of (\ref{eq:By-logX}), then in the same way as above, we would have
arrived at a different tail limit, the GW limit defined by replacing
$\log q$ by $q$ in (\ref{eq:logq-ERV}) (see \citet{De Valk I}).
Its domain of attraction covers a much more limited range of tail
weights. Furthermore, if $F(0)<1$, then the GW limit implies a log-GW
limit (\emph{cf.} the proof of Lemma 3.5.1 in \citet{Laurens  boek}).
Therefore, to ensure that the results of this article are sufficiently
widely applicable, we focus on the log-GW limit. 

The events considered in (\ref{eq:LDP-log-GW}) with $D\subset[0,\infty)$
imply that $X$ is in the interval%
\footnote{Depending on $\theta$, we can extend this somewhat to $[q(y)\textrm{e}^{-cg(y)},\infty)$
for some $c>0$; see (\ref{eq:logGW-in-F}) %
} $[q(y),\infty)$ for $q(y)>0$. In a multivariate setting, it would
be desirable to extend this interval to $\mathbb{R}$, since a multivariate
event could be extreme in one variable, but not in some other variable.
This can be accomplished using a trick: define an approximation $\tilde{q}_{y}$
of $q$ (see (\ref{eq:defq})) for $y\in q^{-1}((0,\infty))$ by
\begin{equation}
\tilde{q}_{y}(z):=\begin{cases}
q(z) & \textrm{if}\: z\leq y\\
q(y)\textrm{e}^{g(y)h_{\theta}(z/y)} & \textrm{if}\: z>y,
\end{cases}\label{eq:q_curl}
\end{equation}
so for $z>y$, $\tilde{q}_{y}(z)$ is the log-GW tail approximation;
for $z\leq y$, it is exact. 

A random variable $Y$ with the standard exponential distribution
satisfies
\begin{equation}
-\inf A^{o}\leq\liminf_{y\rightarrow\infty}\frac{1}{y}\log P(Y\in Ay)\leq\limsup_{y\rightarrow\infty}...\leq-\inf\bar{A}\label{eq:LDP_Y}
\end{equation}
for every Borel set $A\subset[0,\infty)$, which can be proven in
a similar manner as Proposition \ref{pro:PropA}. If $F$ is continuous,
then $Y=-\log(1-F(X))$ has the standard exponential distribution
and $q$ is increasing, so we can substitute $P(X\in q(Ay))$ for
$P(Y\in Ay)$ in (\ref{eq:LDP_Y}). Under the assumptions of Proposition
\ref{pro:PropA}, we can substitute $P(X\in\tilde{q}_{y}(Ay))$ for
$P(Y\in Ay)$ in (\ref{eq:LDP_Y}) as well, extending (\ref{eq:LDP-log-GW})
to
\begin{equation}
-\inf A^{o}\leq\liminf_{y\rightarrow\infty}\frac{1}{y}\log P(X\in\tilde{q}_{y}(Ay))\leq\limsup_{y\rightarrow\infty}...\leq-\inf\bar{A}:\label{eq:LDP_univar_GWmixed}
\end{equation}

\begin{proposition}
\label{pro:PropC}(a) If $F$ is continuous, then (\ref{eq:LDP-log-GW})
for every Borel set $D\subset[0,\infty)$, (\ref{eq:logq-ERV}) and
(\ref{eq:logGW-in-F}) are all equivalent to (\ref{eq:LDP_univar_GWmixed})
for every Borel set $A\subset[0,\infty)$.\end{proposition}
\begin{svmultproof}
Equivalence of (\ref{eq:logq-ERV}), (\ref{eq:logGW-in-F}) and (\ref{eq:LDP-log-GW})
follows from Proposition \ref{pro:PropA}(b). If $A^{o}\cap(-\infty,1)$
is nonempty, then $P(X\in\tilde{q}_{y}(Ay))\geq P(Y\in(A^{o}\cap(-\infty,1))y)$
and the lower bound in (\ref{eq:LDP_univar_GWmixed}) follows from
(\ref{eq:LDP_Y}). If not, then by (\ref{eq:q_curl}), $P(X\in\tilde{q}_{y}(Ay))\geq P(X\in\tilde{q}_{y}(A^{o}y))=P((\log X-\log q(y))/g(y)\in h_{\theta}(A^{o}))$
with $h_{\theta}(A^{o})\subset[0,\infty)$, so Proposition \ref{pro:PropA}
implies the lower bound in (\ref{eq:LDP_univar_GWmixed}). The upper
bound is proven similarly. To show that (\ref{eq:LDP_univar_GWmixed})
implies (\ref{eq:logGW-in-F}) for $x\in h_{\theta}((1,\infty))$,
take $A=[\lambda,\infty)$ for $\lambda\geq1$; it can be extended
to $x\in h_{\theta}((0,\infty))$ by a standard argument. 
\end{svmultproof}

When restricting $A$ to $[1,\infty)$, (\ref{eq:LDP_univar_GWmixed})
is equivalent to (\ref{eq:LDP-log-GW}) for $D=h_{\theta}^{-1}(A)$.
With $A$ in $[0,\infty)$, therefore, (\ref{eq:LDP_univar_GWmixed})
provides the intended generalisation of (\ref{eq:LDP-log-GW}). Note
that the log-GW index $\theta$ and auxiliary function $g$ are now
hidden in the approximation $\tilde{q}_{y}$ in (\ref{eq:q_curl}).
However, they are as essential in (\ref{eq:LDP_univar_GWmixed}) as
they are in the more explicit (\ref{eq:LDP-log-GW}).

\section{\label{sec:Dependence}Bounds and limits for probabilities of multivariate
tail events}

For the univariate tail, we obtained the LDP (\ref{eq:LDP_univar_GWmixed})
in a form which closely resembles (\ref{eq:LDP_Y}) for the standard
exponential distribution. This suggests that for a multivariate generalisation,
we examine first the case of a random vector $Y:=(Y_{1},..,Y_{m})$
with distribution function having standard exponential marginals.
A straightforward multivariate generalisation of the LDP (\ref{eq:LDP_Y})
would be
\begin{equation}
-\inf I(A^{o})\leq\liminf_{y\rightarrow\infty}\frac{1}{y}\log P(Y/y\in A)\leq\limsup_{y\rightarrow\infty}...\leq-\inf I(\bar{A})\label{eq:LDP_mexp}
\end{equation}
for every Borel set $A\subset[0,\infty)^{m}$, with $I$ some rate
function; we may regard (\ref{eq:LDP_mexp}) as the analogue of the
classical expression (\ref{eq:conv_expmeasure}). Further on, we will
prove that (\ref{eq:LDP_mexp}) holds if
\begin{equation}
I(x):=-\inf_{\varepsilon>0}\liminf_{y\rightarrow\infty}\frac{1}{y}\log P(Y/y\in B_{\varepsilon}(x))=-\inf_{\varepsilon>0}\limsup_{y\rightarrow\infty}\frac{1}{y}\log P(Y/y\in B_{\varepsilon}(x)),\label{eq:weakLDP}
\end{equation}
with $B_{\varepsilon}(x):=\{x'\in\mathbb{R}^{m}:\:\bigl\Vert x-x'\bigr\Vert_{\infty}<\epsilon\}$
the open ball of radius $\varepsilon>0$ with centre $x\in\mathbb{R}^{m}$.
For now, we turn to the rate function $I$, defined by (\ref{eq:weakLDP})
as some kind of limiting density, with the probability of an open
ball replaced by its logarithm. Several properties of $I$ follow
immediately from (\ref{eq:weakLDP}) and the exponential marginals
of $Y$. For every $\varepsilon>0$ and $x\in\mathbb{R}^{m}$ with
$x_{j}=\lambda>0$ for some $j\in\{1,..,m\}$, $\frac{1}{y}\log P(Y/y\in B_{\varepsilon}(x))\leq\frac{1}{y}\log P(Y_{j}/y>(\lambda-\varepsilon))=\varepsilon-\lambda$,
so
\begin{equation}
I(x)\geq\max_{j\in\{1,..,m\}}x_{j}\quad\forall x\in\mathbb{R}^{m}.\label{eq:I_lb}
\end{equation}

This implies that $I$ is a \textit{good} rate function, meaning that
$I^{-1}([0,a])$ is compact for every $a\in[0,\infty)$. Also, since
$B_{\varepsilon}(x\lambda)=\lambda B_{\varepsilon/\lambda}(x)$,
\begin{equation}
I(x\lambda)=\lambda I(x)\quad\forall\lambda>0,\: x\in\mathbb{R}^{m}.\label{eq:I_mexp_symm}
\end{equation}

Furthermore, $I(0)=0$, since $P(\left\Vert Y\right\Vert _{\infty}\leq y\varepsilon)\geq1-mP(Y_{1}>\varepsilon y)=1-me^{-\varepsilon y}$
in (\ref{eq:weakLDP}), and $I(x)=\infty$ whenever $\min(x_{1},..,x_{m})<0$.
\begin{remark}
\label{rem:phi}By (\ref{eq:I_mexp_symm}), $I(x)=\varrho(x)I(x/\varrho(x))$
for every $x\in\mathbb{R}^{m}\setminus\{0\}$ and every norm $\varrho$
on $\mathbb{R}^{m}$. This gives for every norm a ``spectral representation''
of $I$, analogous to the spectral measures in classical extreme value
theory (\emph{e.g.} \citet{Laurens  boek}, Section 6.1.4). For example,
in the bivariate case, the rate function can be represented on $[0,\infty)^{2}\setminus\{0\}$
by $I(x)=(x_{1}+x_{2})\psi(x_{2}/(x_{1}+x_{2}))$ with $\psi(t):=I(1-t,t)$
for $t\in[0,1]$, so by (\ref{eq:I_lb}), it satisfies $\psi(t)\geq\max(t,1-t)$
for all $t\in[0,1]$. The similarity of $\psi$ to the dependence
function $A$ of \citet{Pickands-dep} may be misleading, as a rate
function defined by (\ref{eq:weakLDP}) and a distribution function
are very different objects. Besides satisfying $A(t)\geq\max(t,1-t)$
for all $t\in[0,1]$, Pickands' function $A$ is convex, and $A(1)=A(0)=1$.
These latter conditions do not need to apply to $\psi$. \end{remark}
\begin{example}
\label{exa:biv-normal}Let $X\sim\mathcal{N}(0,V)$ with $V$ an $m\times m$
positive-definite matrix with unit diagonal; let $W:=V^{-1}$. Then
\[
I(x)=\Sigma_{j,i\in\{1,..,m\}}w{}_{ji}\sqrt{x_{i}x_{j}},\qquad x\in[0,\infty)^{m}.
\]

In the bivariate case with $v_{12}=v_{21}=:\rho$, $I(x)=(x_{1}+x_{2}-2\rho\sqrt{x_{1}x_{2}})/(1-\rho^{2})$,
so $\psi(t)=(1-2\rho\sqrt{t(1-t)})/(1-\rho^{2})$. If $\rho>0$, then
$I$ is convex and therefore, $\psi$ is convex. Figure \ref{fig:bivar_normal}.1
shows contour plots of $I$ for $\rho=0.8$ (left) and $\rho=0.2$
(middle). On the right, the function $\psi$ is plotted for these
two values of $\rho$; for both, $\psi(1)=\psi(0)>1$.
\end{example}
\begin{figure}
\label{fig:bivar_normal}\includegraphics[scale=0.27]{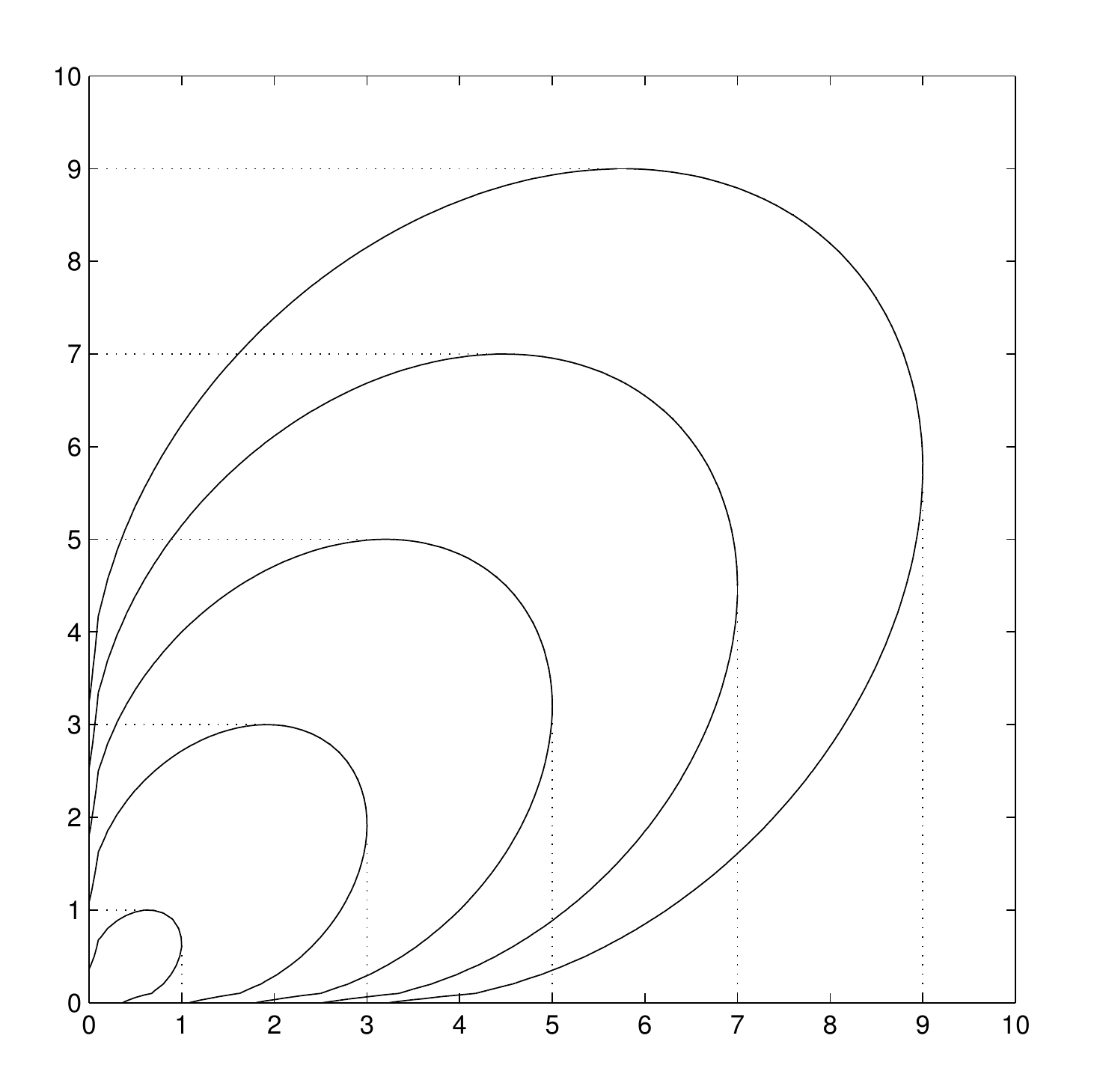}\includegraphics[scale=0.27]{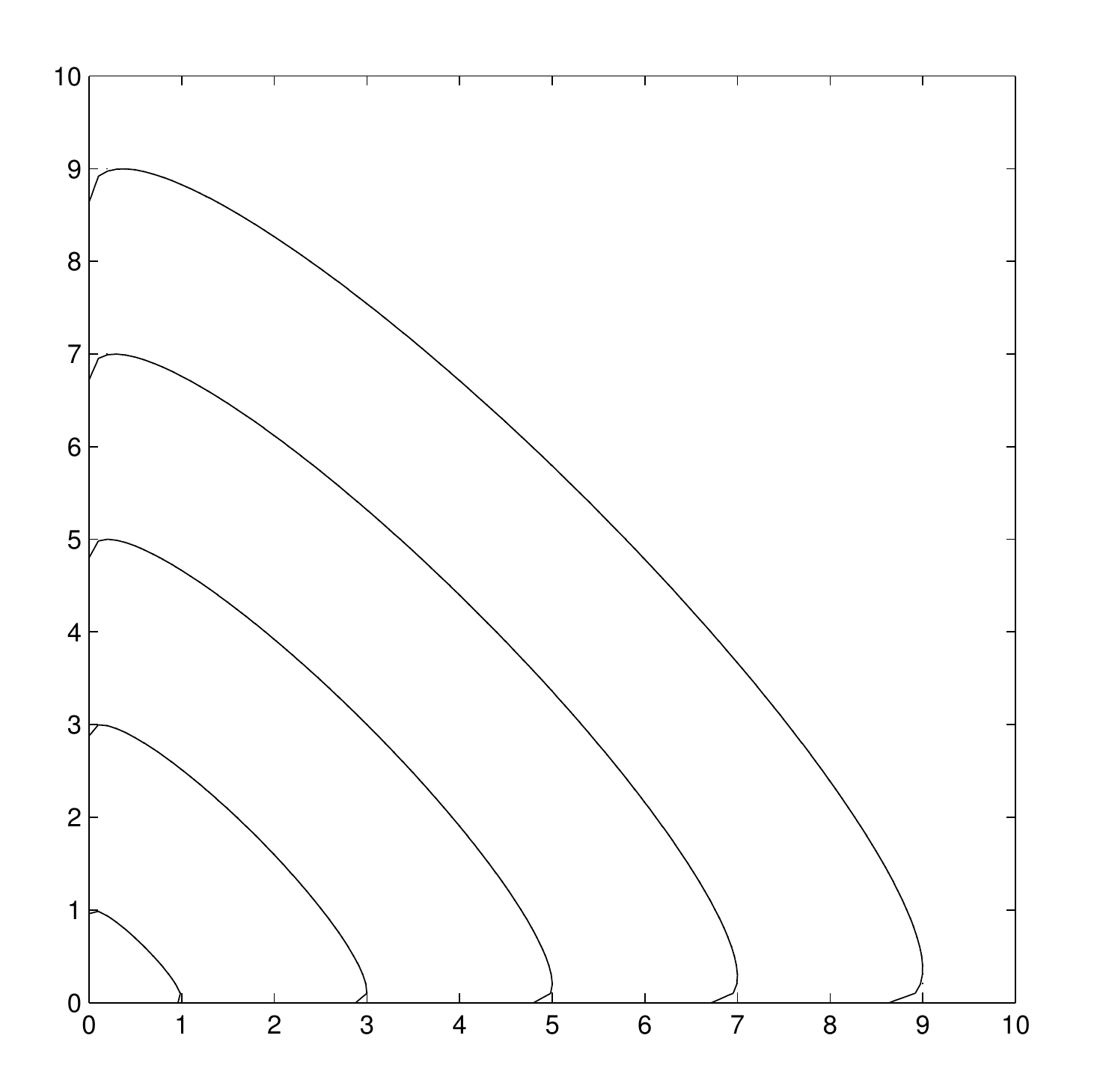}\includegraphics[scale=0.27]{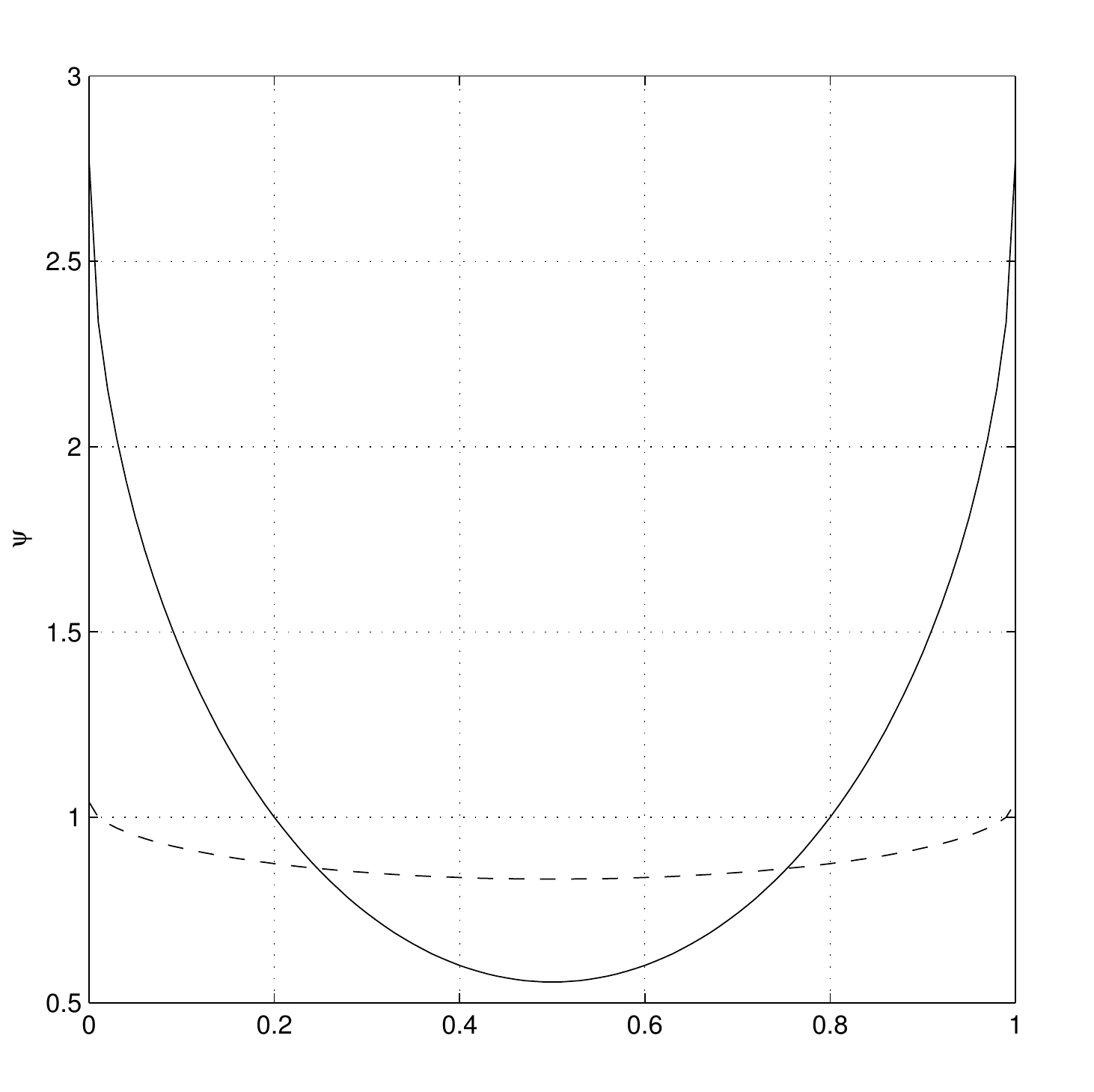}

\protect\caption{Left and middle: contours of the rate function $I$ (drawn) and function
$\kappa$ (dotted; see Section \ref{sec:residual}) for the bivariate
normal distribution with standard marginals and $\rho=0.8$ (left)
and $\rho=0.2$ (middle). Right: function $\psi$ (see text) for $\rho=0.8$
(drawn) and $\rho=0.2$ (dashed).}
\end{figure}

\begin{remark}
If $\lim_{y\rightarrow\infty}P(\min_{j=1,..,m}Y_{j}>y)/P(Y_{1}>y)>0$,
then $I(\mathit{1})=1$ with $\mathit{1}=(1,..,1)$; the converse
is not true. In the bivariate case, $I(\mathit{1})=1$ implies that
$\psi({\scriptstyle \frac{1}{2}})={\scriptstyle \frac{1}{2}}$ for
$\psi$ in Remark \ref{rem:phi}. It does not fix $\psi(t)$ at other
$t\in[0,1]$. For example, for the bivariate normal distribution with
standard marginals (see Example \ref{exa:biv-normal}) and with $\rho=1$,
$\psi(t)=\infty$ for all $t\in[0,1]\setminus\{{\scriptstyle \frac{1}{2}}\}$,
but for a positive mixture of this distribution function with a similar
one with $\rho=r<1$, $\psi(t)=(1-2r\sqrt{t(1-t)})/(1-r^{2})$ for
all $t\in[0,1]\setminus\{{\scriptstyle \frac{1}{2}}\}$.
\end{remark}

\begin{remark}
If $I$ is subadditive, then by (\ref{eq:I_mexp_symm}), it is convex,
and furthermore, by (\ref{eq:I_lb}), it is a norm. Again, this condition
does not need to be satisfied in general. 
\end{remark}
Since $P(Y_{1}>y\alpha)\leq P(\left\Vert Y\right\Vert _{\infty}>y\alpha)\leq mP(Y_{1}>y\alpha)$,
\begin{equation}
\lim_{y\rightarrow\infty}\frac{1}{y}\log P(\left\Vert Y\right\Vert _{\infty}/y>\alpha)=-\alpha\quad\forall\alpha\geq0,\label{eq:sup_limit}
\end{equation}
so the tail of the maximum of $Y_{1},..,Y_{m}$ satisfies the same
limit relation as the tails of each of $Y_{1},..,Y_{m}$ individually.
This implies that the family of probability measures corresponding
to the random variables $\{Y/y,\: y>0\}$ is \emph{exponentially tight}
(\citet{Dembo=000026Zeitouni}): for every $\alpha<\infty$, a compact
$E_{\alpha}\subset\mathbb{R}^{m}$ exists such that
\begin{equation}
\limsup_{y\rightarrow\infty}\frac{1}{y}\log P(Y/y\in E_{\alpha}^{c})<-\alpha,\label{eq:exp_tight}
\end{equation}
which follows from (\ref{eq:sup_limit}) when taking $E_{\alpha}=\{x\in\mathbb{R}^{m}:\:\left\Vert x\right\Vert _{\infty}\leq\alpha+\varepsilon\}$
for some $\varepsilon>0$. As a consequence, 
\begin{theorem}
\label{thm:ldp_exp}If the random vector $Y:=(Y_{1},..,Y_{m})$ with
standard exponential marginals satisfies (\ref{eq:weakLDP}), then
it satisfies (\ref{eq:LDP_mexp}) for all Borel $A\subset[0,\infty)^{m}$
with good rate function $I$ satisfying (\ref{eq:I_mexp_symm}), $I(0)=0$
and the marginal condition
\begin{equation}
\inf_{x\in\mathbb{R}^{m}:\: x_{j}>\lambda}I(x)=\lambda\quad\forall\lambda\geq0,\: j\in\{1,..,m\}.\label{eq:I_marg}
\end{equation}
\end{theorem}
\begin{svmultproof}
By Theorem 4.1.11 in \citet{Dembo=000026Zeitouni}, (\ref{eq:weakLDP})
implies the weak LDP, \emph{i.e}., the lower bound in (\ref{eq:LDP_mexp})
holds for all Borel $A$, and the upper bound of (\ref{eq:LDP_mexp})
holds for all compact $A$. Because of exponential tightness (\ref{eq:exp_tight}),
this implies the LDP (\ref{eq:LDP_mexp}); see Lemma 1.2.18 in \citet{Dembo=000026Zeitouni}.
Then (\ref{eq:I_marg}) follows from (\ref{eq:LDP_mexp}) and the
exponential marginals of $Y$.\end{svmultproof}

\begin{remark}
(\ref{eq:I_lb}) is implied by (\ref{eq:I_marg}). 
\end{remark}
For a continuity set of $I$ satisfying that $\inf I(\bar{A})=\inf I(A^{o})$,
the bounds in (\ref{eq:LDP_mexp}) reduce to a limit:
\begin{equation}
\lim_{y\rightarrow\infty}\frac{1}{y}\log P(Y\in Ay)=-\inf I(A)\quad\forall\lambda>0.\label{eq:LDL}
\end{equation}

A sufficient condition for a set $A$ to be a continuity set of $I$
is that $I$ is continuous and $A\subset\overline{A^{o}}$. Homogeneity
(\ref{eq:I_mexp_symm}) of $I$ allows us to relax this condition:
without assuming continuity of $I$, $A$ is a continuity set if $\inf I(\bar{A})=I(x)$
for some $x\in\overline{A^{o}}\cap\cup_{\lambda>0}(\lambda A^{o})$
($\cup_{\lambda>0}(\lambda A^{o})$ is the smallest cone containing
$A^{o}$). A bivariate example is sketched in Figure \ref{fig:cset}.2.
Let $A^{o}$ be the grey set; if $I$ attains its infimum over $\bar{A}$
on the part of its boundary drawn as a fat line (excluding the points
indicated by circles), then $A$ is a continuity set. In the remainder
of this article, we will discuss continuity sets of rate functions
without considering the particular conditions which make them so.
\begin{figure}[H]
\begin{raggedright}
\includegraphics[scale=0.25]{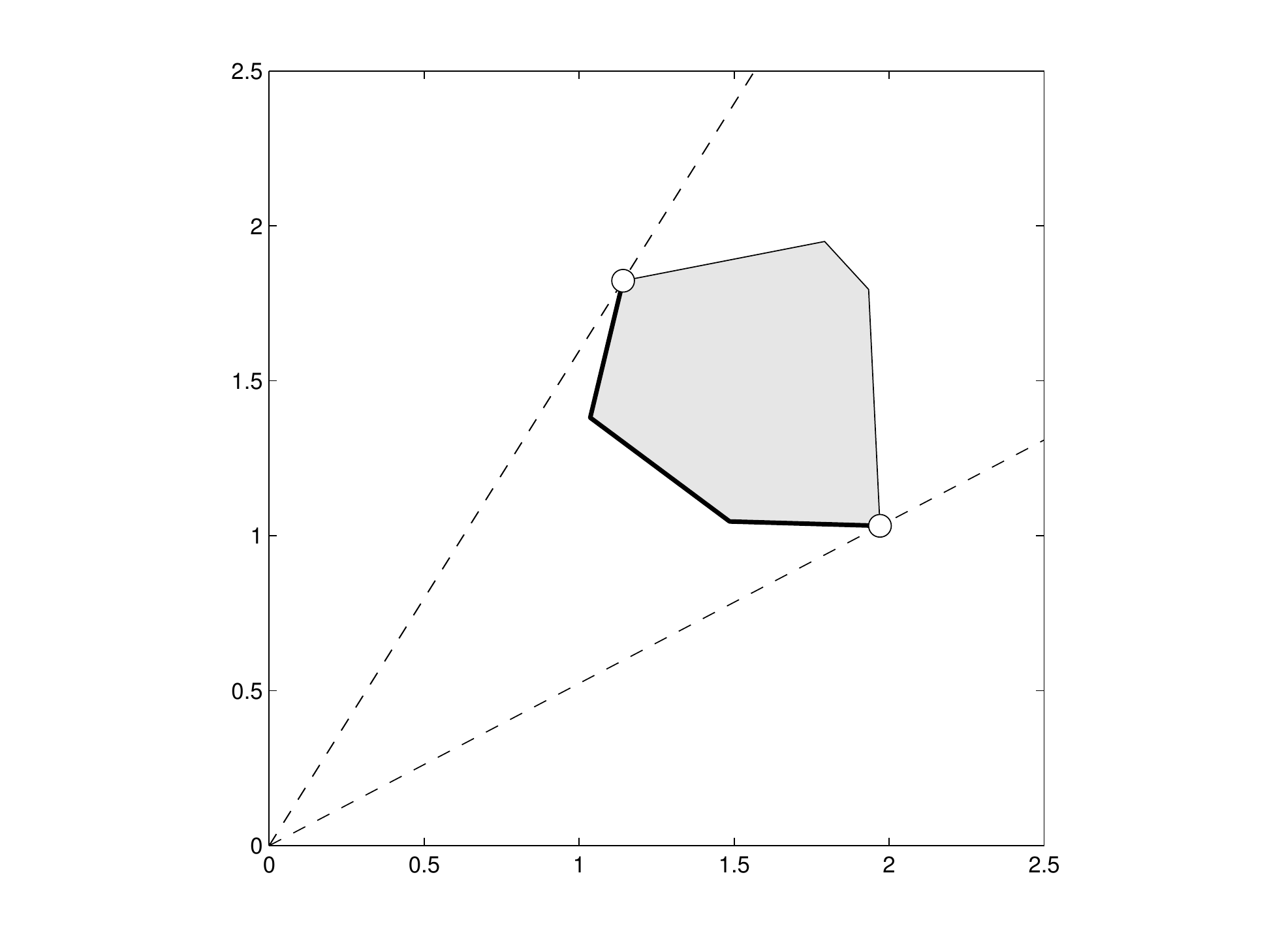}\label{fig:cset}
\par\end{raggedright}

\raggedright{}\protect\caption{Illustration of a continuity set of $I$ (see main text)}
\end{figure}

It is straightforward to extend Theorem \ref{thm:ldp_exp} to a random
vector $X$ with a distribution function $F$ having continuous marginals
$F_{1},..,F_{m}$. As in (\ref{eq:defq}), let for $i=1,..,m$,
\begin{equation}
q_{i}:=F_{i}^{-1}(1-e^{-\textrm{Id}})\label{eq:defq_i}
\end{equation}
and for every $x\in[0,\infty)^{m}$,
\begin{equation}
Q(x):=(q_{1}(x_{1}),..,q_{m}(x_{m})).\label{eq:Qdef}
\end{equation}

Let $Y:=(Y_{1},..,Y_{m})$ with for $j=1,..,m$,
\begin{equation}
Y_{j}:=-\log(1-F_{j}(X_{j})),\label{eq:Ydef}
\end{equation}
so $Y=Q^{-1}(X).$ Because $F_{1},..,F_{m}$ are continuous, $Y$
has exponential marginals. Almost surely, $X=Q(Y)$ with $Q$ defined
by (\ref{eq:Qdef}) and (\ref{eq:defq_i}). Since $Q$ is injective,
$P(X\in Q(yA))=P(Y\in yA)$, so (\ref{eq:LDP_mexp}) is equivalent
to 
\begin{equation}
-\inf I(A^{o})\leq\liminf_{y\rightarrow\infty}\frac{1}{y}\log P(X\in Q(yA))\leq\limsup_{y\rightarrow\infty}...\leq-\inf I(\bar{A}).\label{eq:LDP_mexp_Q}
\end{equation}

Having obtained a multivariate version of (\ref{eq:LDP_Y}), we are
now ready to generalise the univariate tail LDP (\ref{eq:LDP-log-GW})
and its extension (\ref{eq:LDP_univar_GWmixed}) to the multivariate
context. Concerning the latter, one would expect its multivariate
generalisation to be like (\ref{eq:LDP_mexp_Q}), with $Q$ replaced
by an approximation. Let $F_{1},..,F_{m}$ satisfy log-GW tail limits
with scaling functions $g_{1},..,g_{m}$ and log-GW indices $\theta_{1},..,\theta_{m}$,
respectively. As in (\ref{eq:q_curl}), define marginal quantile approximations
\begin{equation}
\tilde{q}_{j,y}(z):=\begin{cases}
q_{j}(z) & \textrm{if}\: z\leq y\\
q_{j}(y)e^{g_{_{j}}(y)h_{\theta_{j}}(z/y)} & \textrm{if}\: z>y
\end{cases}\label{eq:qtilde-def}
\end{equation}
for $y\in\cap_{j=1}^{m}q_{j}^{-1}((0,\infty))$, and let for all $x=(x_{1},..,x_{m})\in[0,\infty)^{m}$,
\begin{equation}
\tilde{Q}_{y}(x):=(\tilde{q}_{1,y}(x_{1}),..,\tilde{q}_{m,y}(x_{m})).\label{eq:Qtilde_def}
\end{equation}

\begin{theorem}
\noindent \label{thm:ldp_complete}Let the random vector $X=(X_{1},..,X_{m})$
have distribution function $F$ with continuous marginals $F_{1},..,F_{m}$
having positive endpoints. 

\noindent (a) If $Y$ defined by (\ref{eq:Ydef}) satisfies (\ref{eq:weakLDP})
and the marginals satisfy log-GW tail limits (\ref{eq:logq-ERV})
with $q=q_{j}$, $g=g_{j}$ and $\theta=\theta_{j}$ for $j=1,..,m$,
then $X$ satisfies
\begin{equation}
-\inf I(A^{o})\leq\liminf_{y\rightarrow\infty}\frac{1}{y}\log P(X\in\tilde{Q}_{y}(yA))\leq\limsup_{y\rightarrow\infty}...\leq-\inf I(\bar{A})\label{eq:LDP_mexp_approx}
\end{equation}

\noindent for every Borel set $A\subset[0,\infty)^{m}$, with $\tilde{Q}_{y}$
given by (\ref{eq:qtilde-def}) and (\ref{eq:Qtilde_def}), and $I$
a good rate function satisfying (\ref{eq:I_mexp_symm}), (\ref{eq:I_marg})
and $I(0)=0.$

\noindent (b) If $X$ satisfies (\ref{eq:LDP_mexp_approx}) for every
Borel set $A\subset[0,\infty)^{m}$ with $\tilde{Q}_{y}$ given by
(\ref{eq:qtilde-def}) and (\ref{eq:Qtilde_def}) and with rate function
$I$ satisfying (\ref{eq:I_marg}), then the marginals satisfy log-GW
tail limits, and $Y$ defined by (\ref{eq:Ydef}) satisfies (\ref{eq:LDP_mexp})
with good rate function $I$ satisfying (\ref{eq:I_mexp_symm}), (\ref{eq:I_marg})
and $I(0)=0.$
\end{theorem}
The proof can be found in Subsection \ref{sub:proof_ldp_complete}.
\begin{remark}
This theorem justifies viewing (\ref{eq:LDP_mexp}) as representation
of tail dependence within the context of the LDP (\ref{eq:LDP_mexp_approx}),
which also represents the marginal tails. The relationship between
the LDPs (\ref{eq:LDP_mexp_approx}) and (\ref{eq:LDP_mexp}) is the
large deviations analogue of a similar relationship in classical extreme
value theory; compare \emph{e.g.} \citet{Resnick book}, Propositions
5.10 and 5.15. 
\end{remark}
From the multivariate generalisation of (\ref{eq:LDP_univar_GWmixed}),
we can now also derive a multivariate version of (\ref{eq:LDP-log-GW}),
equivalent to the restriction of (\ref{eq:LDP_mexp_approx}) to $A\subset[1,\infty)^{m}$:
\begin{corollary}
\label{cor:ldp_complete}Let $\theta:=(\theta_{1},..,\theta_{m})$
and $H_{\theta}(z):=(h_{\theta_{1}}(z_{1}),..,h_{\theta_{m}}(z_{m}))$
for all $z\in(0,\infty)^{m}$. Then (\ref{eq:LDP_mexp_approx}) implies
for every Borel set $D\in[0,\infty)^{m}$:
\[
-\inf I(H_{\theta}^{-1}(D^{o}))
\]
\begin{equation}
\leq\liminf_{y\rightarrow\infty}\frac{1}{y}\log P\left(\left(\frac{\log X_{1}-\log q_{1}(y)}{g_{1}(y)},..,\frac{\log X_{m}-\log q_{m}(y)}{g_{m}(y)}\right)\in D\right)\label{eq:LDP_mexp_tail}
\end{equation}
\[
\leq\limsup_{y\rightarrow\infty}...\leq-\inf I(H_{\theta}^{-1}(\bar{D})).
\]
\end{corollary}
\begin{proof}
See Subsection \ref{sub:proof_ldp_complete}. 
\end{proof}

Note that (\ref{eq:LDP_mexp_tail}) only addresses events within $(q_{1}(y),\infty)\times..\times(q_{m}(y),\infty)$,
which is ``covered'' by all marginal log-GW tail approximations
simultaneously. Just as (\ref{eq:LDP-log-GW}), it can be extended
somewhat. However, the main interest of (\ref{eq:LDP_mexp_tail})
is that it shows the multivariate tail LDP explicitly as a pair of
asymptotic bounds for the probabilities of extreme events defined
in terms of affinely normalised logarithms of the components of $X$.
For applications in statistics, (\ref{eq:LDP_mexp_approx}) should
be more useful, as it applies also to events which are not simultaneously
extreme in every component of $X$.

\section{\label{sec:residual} A connection to residual tail dependence and
related models}

In this section, we digress from the main storyline to examine an
interesting connection between the theory of Section \ref{sec:Dependence}
and earlier work on residual tail dependence (RTD) or hidden regular
variation, introduced in \citet{Ledford=000026Tawn96,Ledford=000026Tawn97,Ledford=000026Tawn98}
and studied in depth in \citet{Resnick2002}, amongst others. In the
bivariate case, RTD offers a model of tail dependence within the classical
domain of asymptotic independence of component-wise maxima (\emph{e.g.}
\citet{Laurens  boek}, Section 7.6). For a random vector $X$ on
$\mathbb{R}^{m}$ with continuous marginals $F_{1},...,F_{m}$, defining
the random vector $V:=(V_{1},..,V_{m})$ with standard Pareto-distributed
variables by (\ref{eq:Ydef}), one way to describe RTD is that for
some positive function $S$ on $(0,\infty)^{m}$,

\begin{equation}
\lim_{t\rightarrow\infty}\frac{P(V_{j}>tx_{j}\;\forall j\in\{1,..,m\})}{P(V_{j}>t\;\forall j\in\{1,..,m\})}=:S(x)>0\label{eq:rtd}
\end{equation}
for all $x\in(0,\infty)^{m}$. The limiting function $S$ satisfies
$S(\mathit{1})=1$, with $\mathit{1}$ the vector in $\mathbb{R}^{m}$
with all its components equal to $1$. Furthermore, the denominator
in (\ref{eq:rtd}) must be regularly varying, so $S(\mathit{1}\lambda)=\lambda^{\nicefrac{-1}{\eta}}$
for all $\lambda>0$ with $\eta\in(0,1]$ the \textit{residual dependence
index, }and by (\ref{eq:rtd}),
\begin{equation}
S(x\lambda)=\lambda^{\nicefrac{-1}{\eta}}S(x)\quad\forall x\in(0,\infty)^{m},\:\lambda>0.\label{eq:RTD_symm}
\end{equation}

Every regularly varying function $f\in RV_{\alpha}$ can be represented
as
\begin{equation}
f(y)=c(y)e^{\int_{y_{0}}^{y}a(t)t^{-1}dt}\label{eq:rep-RV}
\end{equation}
with $c(y)\rightarrow c_{0}>0$ and $a(y)\rightarrow\alpha$ as $y\rightarrow\infty$.
A minor strengthening of regular variation is that $f$ satisfies
the \emph{Von Mises condition} (see \emph{e.g.} Proposition 1.15 of
\citet{Resnick book}), which means that $c$ in (\ref{eq:rep-RV})
can be taken equal to a positive number $c_{0}$; it implies that
$f$ is differentiable with derivative $f'(y)=a(y)f(y)/y$. Note that
whenever the LDP (\ref{eq:LDP_mexp}) holds and $\inf I(A)\in(0,\infty)$
for a particular Borel set $A$, then the function $(y\mapsto-\log P(Y/y\in A))$
is in $RV_{1}$. Therefore, within the context of the LDP (\ref{eq:LDP_mexp}),
the statement that $(y\mapsto-\log P(Y/y\in A))$ satisfies the Von
Mises condition makes sense as a smoothness condition. The following
relates RTD to the tail LDP (\ref{eq:LDP_mexp}). 
\begin{proposition}
\label{pro:hidden}(a) RTD (\ref{eq:rtd}) implies
\begin{equation}
\lim_{y\rightarrow\infty}\frac{1}{y}\log P(Y/y\in(\lambda,\infty)^{m})=-\lambda/\eta\quad\forall\lambda>0,\label{eq:quasiLDP}
\end{equation}
with $\eta$ the residual dependence index of $X$.

\noindent (b) If $X$ satisfies an LDP (\ref{eq:LDP_mexp}) with the
function $(y\mapsto-\log P(Y/y\in(1,\infty)^{m}))$ satisfying the
Von Mises condition, then (\ref{eq:rtd}) holds for $x=\mathit{1}\lambda$
for all $\lambda>0$ with $S(\mathit{1}\lambda)=\lambda^{\nicefrac{-1}{\eta}}$
and $\eta=1/I(\mathit{1})$. \end{proposition}
\begin{svmultproof}
Define $Y_{_{\wedge}}:=\min_{j\in\{1,..,m\}}Y_{j}$, and let $H_{_{\wedge}}$
be the distribution function of $Y_{_{\wedge}}$. By (\ref{eq:rtd}),
the survival function $1-H_{_{\wedge}}\circ\log$ of the random variable
$\exp Y_{_{\wedge}}$ is regularly varying with index $-1/\eta$.
Therefore, $f:=1/(1-H_{_{\wedge}}\circ\log)\in RV_{\{1/\eta\}}$,
so by the Potter bounds (\citet{Bingham}), for every $\varepsilon\in(0,1/\eta),$
there is $z_{\varepsilon}>0$ such that $(1-\varepsilon)(x/z)^{\nicefrac{1}{\eta}-\varepsilon}\leq f(x)/f(z)\leq(1+\varepsilon)(x/z)^{\nicefrac{1}{\eta}+\varepsilon}$
for all $z\geq z_{\varepsilon}$ and $x\geq z$. Taking logarithms
and substituting $\textrm{e}^{y\lambda}$ for $x$ gives $\lim_{y\rightarrow\infty}y^{-1}\log f(\textrm{e}^{y\lambda})\rightarrow\lambda/\eta$
for all $\lambda>0$, so (\ref{eq:quasiLDP}) follows. For (b), note
that due to (\ref{eq:I_mexp_symm}), the LDP (\ref{eq:LDP_mexp})
implies (\ref{eq:quasiLDP}) with $\eta=1/I(\mathit{1})$, so $w(y):=-\log(1-H_{_{\wedge}}(y))\sim y/\eta$
as $y\rightarrow\infty$. Therefore, since $w$ satisfies the Von
Mises condition, $w'(y)\rightarrow1/\eta$ and by averaging, $w(y+r)-w(y)\rightarrow r/\eta$
as $y\rightarrow\infty$ for every $r\in\mathbb{R}$. This is equivalent
to (\ref{eq:rtd}) for $x=\mathit{1}\lambda$ with $S(\mathit{1}\lambda)=\lambda^{\nicefrac{-1}{\eta}}$
for every $\lambda>0$. 
\end{svmultproof}

Proposition \ref{pro:hidden} shows that RTD implies a limited LDP-like
condition and in turn, the LDP (\ref{eq:LDP_mexp}) with an additional
smoothness condition implies an RTD-like condition. 
\begin{example}
\label{exa:rtd}The bivariate normal $X$ of Example \ref{exa:biv-normal}
satisfies the conditions for Proposition \ref{pro:hidden}(b) with
$I(\mathit{1})=2/(1+\rho)$. Indeed, (\ref{eq:rtd}) holds for $x=\mathit{1}\lambda$
for all $\lambda>0$, with $S(\mathit{1}\lambda)=\lambda^{\nicefrac{-1}{\eta}}$
and $\eta=(1+\rho)/2$; see Example Class 2(1) in \citet{Ledford=000026Tawn96}. 
\end{example}
If $\lim_{t\rightarrow\infty}t^{-1}P(V_{j}>t\;\forall j\in\{1,..,m\})=0$,
then there is a discrepancy between the ``hidden'' regularity of
the survival function in $(0,\infty)^{m}$ described by (\ref{eq:rtd})
and the regularity of the marginals. In contrast, the LDP (\ref{eq:LDP_mexp})
provides a single consistent description of the multivariate tail
which includes the marginal tails. Furthermore, the next theorem shows
that under a smoothness assumption similar to the one in Proposition
\ref{pro:hidden}(b), the LDP (\ref{eq:LDP_mexp}) implies a useful
extension of RTD. Let for all $a\in\mathbb{R}^{m}$,
\begin{equation}
A_{a}:=\{x\in\mathbb{R}^{m}:\: x_{j}>a_{j}\:\forall j\in\{1,..,m\}\}.\label{eq:A_a}
\end{equation}

\begin{theorem}
\label{thm:RD_general}(a) Assume that the LDP (\ref{eq:LDP_mexp})
applies. To any Borel set $A\subset\mathbb{R}^{m}$ which is a continuity
set of $I$ with $(y\mapsto-\log P(Y/y\in A))$ satisfying the Von
Mises condition, the following limit relation applies:
\begin{equation}
\lim_{t\rightarrow\infty}\frac{P(Y\in A\log(t\lambda))}{P(Y\in A\log t)}=\lambda^{-\inf I(A)}\quad\forall\lambda>0,\label{eq:short-range-limit}
\end{equation}
with $I$ satisfying (\ref{eq:I_mexp_symm}), (\ref{eq:I_marg}) and
$I(0)=0$. In particular, for every $a\in[0,\infty)^{m}$ such that
the function $(y\mapsto-\log P(Y_{j}>ya_{j}\;\forall j\in\{1,..,m\})$
satisfies the Von Mises condition,
\begin{equation}
\lim_{t\rightarrow\infty}\frac{P(V_{j}>(t\lambda)^{a_{j}}\;\forall j\in\{1,..,m\})}{P(V_{j}>t{}^{a_{j}}\;\forall j\in\{1,..,m\})}=\lambda^{-\inf I(A_{a})}\quad\forall\lambda>0.\label{eq:short-range-SF}
\end{equation}

(b) Eq. (\ref{eq:short-range-limit}) with $\inf I(A)\in(0,\infty)$
implies (\ref{eq:LDL}).\end{theorem}
\begin{svmultproof}
For $A$ a continuity set of $I$, (\ref{eq:LDP_mexp}) implies (\ref{eq:LDL}),
and (\ref{eq:short-range-limit}) is obtained in the same manner as
in the proof of Proposition \ref{pro:hidden}(b). In particular, $A_{a}$
is a continuity set of $I$ for every $a\in[0,\infty)^{m}$. Therefore,
substituting $A_{a}$ for $A$ in (\ref{eq:short-range-limit}), we
obtain (\ref{eq:short-range-SF}). This proves (a). For (b), note
that $f:=(t\mapsto1/P(Y\in A\log t))\in RV_{\inf I(A)}$. Therefore,
just as in the proof of Proposition \ref{pro:hidden}(a), $\lim_{y\rightarrow\infty}y^{-1}\log f(\textrm{e}^{y\lambda})\rightarrow\lambda\inf I(A)$
for all $\lambda\geq1$, which implies (\ref{eq:LDL}).
\end{svmultproof}

Combining (a) and (b) in Theorem \ref{thm:RD_general}, we see that
under the Von Mises condition (for $A$ a Borel continuity set of
$I$), the limit relation (\ref{eq:short-range-limit}) for a probability
ratio, and the limit relation (\ref{eq:LDL}) of the normalised logarithm
of a probability are equivalent. 

In the special case of $a=\mathit{1}$, (\ref{eq:short-range-SF})
becomes equivalent to (\ref{eq:rtd}) with $x=\mathit{1}\lambda$
and $\eta=1/I(\mathit{1})$, so on the diagonal, (\ref{eq:short-range-SF})
and RTD (\ref{eq:rtd}) agree; elsewhere, they differ. Defining a
function $\kappa$ by $\kappa(a):=\inf I(A_{a})$ for every $a\in[0,\infty)^{m}$,
(\ref{eq:short-range-SF}) becomes identical to an extension of RTD
recently introduced in \citet{Wadsworth=000026Tawn}. \citet{Wadsworth=000026Tawn}
proposed this assumption to close the possible gap between (\ref{eq:rtd})
and the regularity of the marginal tails. It is curious that this
condition, requiring the existence of separate limits of the survival
function along chosen paths, is derivable from the simple LDP (\ref{eq:LDP_mexp}).

The generalisation of (\ref{eq:short-range-SF}) with $\inf I(A_{a})$
replaced by $\kappa(a)$ to the apparently new limit relation (\ref{eq:short-range-limit})
is not trivial. Another generalisation, proposed in \citet{Wadsworth=000026Tawn},
is
\begin{equation}
\lim_{t\rightarrow\infty}\frac{P(Y\in B+a\log(t\lambda))}{P(Y\in B+a\log t)}=\lambda^{-\kappa(a)},\label{eq:WandT_2}
\end{equation}
derived in Section 3.3 of \citet{Wadsworth=000026Tawn} for the bivariate
case under the assumption that $\kappa$ is differentiable and $a\in[0,\infty)^{m}\setminus\{0\}$
satisfies $\partial\kappa(a)/\partial a_{j}>0$ for $j=1,..,m$. As
noted in \citet{Wadsworth=000026Tawn}, $a$ in (\ref{eq:WandT_2})
would have to be chosen in an application. This would be no problem
if the choice did not matter. However, the limiting behaviour of the
probability of the event $Y\in B+a\log t$ as $t\rightarrow\infty$
is determined by $a$ in (\ref{eq:WandT_2}); not by $B$. Therefore,
for estimating probabilities of extreme events, (\ref{eq:short-range-limit})
seems more promising than the local limits (\ref{eq:WandT_2}) for
chosen $a.$

In (\ref{eq:short-range-limit}), it is not $\kappa$, but the rate
function $I$ which determines the attenuation rate. For any $a\in[0,\infty)^{m}$,
$I(a)$ and $\kappa(a)$ are identical only if $I(a+x)\geq I(a)$
for all $x\in[0,\infty)^{m}$. This condition is rather restrictive,
as a rate function resembles a density more than it resembles a survival
function; see definition (\ref{eq:weakLDP}). 
\begin{example}
\label{exa:kappa}As an illustration, let $X$ be bivariate normal
with correlation coefficient $\rho$ as in Example \ref{exa:biv-normal}
(Section \ref{sec:Dependence}). By Example 1(a,b) in Table 1 of \citet{Wadsworth=000026Tawn},
$\kappa(x)=I(x)$ if $\min(x_{1}/x_{2},x_{2}/x_{2})>\rho^{2}$ or
if $\rho<0$ and $\min(x_{1},x_{2})>0$, and $\kappa(x)=\max(x_{1},x_{2})$
for all other $x\in[0,\infty)^{2}$. The left and middle panels of
Figure \ref{fig:bivar_normal}.1 display contours of $\kappa$ overlaying
the contours of $I$ for $\rho=0.8$ and $\rho=0.2$. For $\rho=0.2$,
contours of $\kappa$ largely overlap with those of $I$; for $\rho=0.8$,
there are wide zones where the contours of $\kappa$ and $I$ differ. 
\end{example}

\section{\label{sec:estimator}A simple estimator for very small probabilities}

We are now going to apply the theory of Section \ref{sec:Dependence}
to the problem of estimation of probabilities of extreme events $p_{n}$
satisfying (\ref{eq:p_n}) from $X^{(1)},...,X^{(n)}$, with $X^{(1)},X^{(2)},...$
a sequence of \emph{iid} copies of a random vector $X$ in $\mathbb{R}^{m}$
with distribution function $F$ having continuous marginals $F_{1},..,F_{m}$.
Denoting the underlying probability space as $(\Omega,\mathcal{F},\mathbb{\mathcal{P}})$,
let $\mathcal{F}_{n}\subset\mathcal{F}$ be the $\sigma$-algebra
generated by $X^{(1)},...,X^{(n)}$. 

Generalising (\ref{eq:p_n}) to $\tau_{2}>\tau_{1}>0$, consider events
of the form $B_{n}:=Q(A\log n)$ with $A\subset[0,\infty)^{m}$ and
$Q$ given by (\ref{eq:Qdef}). Suppose that the tail LDP (\ref{eq:LDP_mexp})
applies. Then for every Borel set $A\subset[0,\infty)^{m}$ which
is a continuity set of $I$ satisfying that $\inf I(A)\in(0,\infty)$,
we have: $-\log P(X\in B_{n})=-\log P(Y\in A\log n)\sim(\log n)\inf I(A)$
and $-\log P(Q(Y/\ell)\in B_{n})=-\log P(Y\in A\ell\log n)\sim\ell(\log n)\inf I(A)$
for all $\ell>0$, so
\begin{equation}
\log P(X\in B_{n})\sim\ell^{-1}\log P(Q(Y/\ell)\in B_{n})\quad\forall\ell>0.\label{eq:P_approx}
\end{equation}

This suggests estimating the left-hand side of (\ref{eq:P_approx})
by replacing $Q$ on the right-hand side by an estimator $\hat{Q}_{n}$
and $Y$ by an estimator $\hat{Y}_{n}$, and then choosing $\ell$
small enough that $P(\hat{Q}_{n}(\hat{Y}_{n}/\ell)\in B_{n})$ can
be estimated directly from the data by counting. 

Estimation of $Q$ boils down to a univariate quantile estimation
problem, so we will proceed to examine this first. Assume that every
marginal satisfies a log-GW tail limit (\emph{i.e}., the univariate
tail LDP, see Section \ref{sec:univar}). Let $X_{j,1:n}\leq...\leq X_{j,n:n}$
be the marginal order statistics derived from the marginal sample
$X_{j}^{(1)},...,X_{j}^{(n)}$. For some intermediate sequence $(k_{n})$
and for $n$ large enough that $X_{j,n-k_{n}+1:n}>0$ for $j=1,..,m$,
define the following estimator $\hat{q}_{j,n}$ for $q_{j}$ (compare
(\ref{eq:qtilde-def})):
\begin{equation}
\hat{q}_{j,n}(z):=\begin{cases}
X_{j,\left\lfloor n(1-\textrm{e}^{-z})\right\rfloor +1:n} & \textrm{if\:}z\in[0,y_{n}]\\
X_{j,n-k_{n}+1:n}\exp\left(\hat{g}_{j,n}h_{\hat{\theta}_{j,n}}(z/y_{n})\right) & \textrm{if\:}z>y_{n}
\end{cases}\label{eq:q_hat}
\end{equation}
with
\begin{equation}
y_{n}:=\log(n/k_{n}).\label{eq:def_y_n}
\end{equation}

For $z>y_{n}$, $\hat{q}_{j,n}(z)$ follows a log-GW tail with $\hat{\theta}_{j,n}$
and $\hat{g}_{j,n}$ estimators for $\theta_{j}$ and $g_{j}$ in
(\ref{eq:qtilde-def}), respectively; for other $z$, the empirical
quantile is used as estimator. The only assumption we will make on
the quantile estimator is that the probability-based quantile estimation
error $\hat{\nu}_{j,n}$, defined by
\begin{equation}
\hat{\nu}_{j,n}(z):=\frac{\log(1-F_{j}(\hat{q}_{j,n}(z)))}{\log(1-F_{j}(q_{j}(z)))}-1=z^{-1}q_{j}^{-1}\hat{q}_{j,n}(z)-1\label{eq:nu_def}
\end{equation}
for $z\geq0$, satisfies
\begin{equation}
\lim_{n\rightarrow\infty}\sup_{\lambda\in[1,\Lambda]}\left|\hat{\nu}_{j,n}(y_{n}\lambda)\right|=0\quad a.s.\quad\forall\Lambda>1,\: j=1,..,m.\label{eq:nu_conv-0}
\end{equation}

Estimators $\hat{\theta}_{n,j}$ and $\hat{g}_{n,j}$ in (\ref{eq:q_hat})
satisfying this requirement were considered in \citet{De Valk I}.
Let 
\begin{equation}
\hat{Q}_{n}(x):=(\hat{q}_{1,n}(x_{1}),...,\hat{q}_{m,n}(x_{m}))\label{eq:Q_hat}
\end{equation}
for every $x\in[0,\infty)^{m}$. Define the following estimator for
$Y_{j}^{(i)}:=-\log(1-F(X_{j}^{(i)}))$:
\begin{equation}
\hat{Y}_{j,n}^{(i)}:=-\log(1-(R_{j,n}^{(i)}-{\scriptstyle \frac{1}{2}})/n)\label{eq:Zhat_def_curly}
\end{equation}
with $R_{j,n}^{(i)}:=\sum_{l=1}^{n}\mathbf{1}(X_{j}^{(l)}\leq X_{j}^{(i)})$
the marginal rank of $X_{j}^{(i)}$. 

For every $n$-tuple of events $\mathcal{C}_{n}:=(\mathcal{C}_{n}^{(1)},..,\mathcal{C}_{n}^{(n)})$
satisfying $\mathcal{C}_{n}^{(i)}\in\mathcal{F}_{n}$ for $i=1,..,n$,
define the ``empirical probability'' $\hat{p}_{n}(\mathcal{C}_{n}):=\omega\mapsto\hat{p}_{n}(\mathcal{C}_{n})(\omega)$
on $\Omega$ by
\begin{equation}
\hat{p}_{n}(\mathcal{C}_{n})(\omega):=n^{-1}\sum_{i=1}^{n}\mathbf{1}(\omega\in\mathcal{C}_{n}^{(i)}).\label{eq:p_hat_def}
\end{equation}

For some $\xi>0$, determine a value of the analogue of $\ell$ in
(\ref{eq:P_approx}) as
\begin{equation}
\ell_{n}^{+}(B):=\sup\{l>0:\:\hat{p}_{n}(\hat{Q}_{n}(\hat{Y}_{n}/l)\in B)\geq(k_{n}/n)^{\xi}\},\label{eq:l+def}
\end{equation}
with $\sup\{\emptyset\}:=0$ and with $\hat{p}_{n}(\hat{Q}_{n}(\hat{Y}_{n}/l)\in B)=n^{-1}\sum_{i=1}^{n}\mathbf{1}(\hat{Q}_{n}(\hat{Y}_{n}^{(i)}/l)\in B)$
in accordance with (\ref{eq:p_hat_def}). 

Let $\pi$ denote the probability measure corresponding to $F$. Now
consider the following estimator for $\pi(B):=P(X\in B)$:
\begin{equation}
{\displaystyle \hat{\pi}_{n}^{\mathrm{I}}(B):=\left(k_{n}/n\right)^{\xi/\ell_{n}^{+}(B)}}.\label{eq:estimator-I}
\end{equation}

If $B_{n,\tau}:=Q(A\tau\log n)$ is substituted for $B$, then under
mild restrictions on $A$ and $(k_{n})$, this estimator converges
in the large deviation sense for all $\tau>0$:
\begin{theorem}
\label{thm: estimator-I}Let $X^{(1)},X^{(2)},...$ be iid copies
of a random vector $X$ on $\mathbb{R}^{m}$ satisfying the conditions
of Theorem \ref{thm:ldp_complete}(a), including continuous marginals
satisfying log-GW tail limits.  For a sequence  $(k_{n})$ satisfying
\begin{equation}
0\leq c':=\liminf_{n\rightarrow\infty}\frac{\log k_{n}}{\log n}\leq\limsup_{n\rightarrow\infty}\frac{\log k_{n}}{\log n}=:c<1,\label{eq:k0_for_estimator}
\end{equation}
consider the estimator (\ref{eq:estimator-I}) for $P(X\in B)$, with
the quantile estimator (\ref{eq:q_hat}) satisfying (\ref{eq:nu_conv-0})
and with $\xi\in(0,(1-c')^{-1})$. Then for $B_{n,\tau}:=Q(A\tau\log n)$,
with $A\subset[0,\infty)^{m}$ any Borel set which is a continuity
set of $I$ defined by (\ref{eq:weakLDP}) and (\ref{eq:Ydef}) and
satisfies $\inf I(A)\in(0,\infty)$,
\begin{equation}
\lim_{n\rightarrow\infty}\sup_{\tau\in[T^{-1},T]}\left|\frac{\log\hat{\pi}_{n}^{\mathrm{I}}(B_{n,\tau})}{\log P(X\in B_{n,\tau})}-1\right|=0\quad a.s.\quad\forall T>1.\label{eq:claim_estimator-I}
\end{equation}

\end{theorem}
The proof can be found in Subsection \ref{sub:Proof-P-estimator-1}. 
\begin{remark}
By (\ref{eq:LDP_mexp_Q}), as $\inf I(A^{o})=\inf I(\bar{A})$, $P(X\in B_{n,\tau})=n^{-\tau\inf I(A)(1+o(1))}$
in (\ref{eq:claim_estimator-I}), so the probability range (\ref{eq:p_n})
is covered by Theorem \ref{thm: estimator-I}.
\end{remark}
In practice, computing or approximating (\ref{eq:l+def}) may not
be easy; for example, in engineering applications, it may involve
running a complex numerical model for every datapoint. Therefore,
it would be an advantage to replace $\ell_{n}^{+}(B)$ in (\ref{eq:estimator-I})
by an arbitrary value in some suitable interval. Define for some $\vartheta\in(0,\xi]$
\begin{equation}
\ell_{n}^{-}(B):=\sup\{l>0:\:\hat{p}_{n}(\hat{Q}_{n}(\hat{Y}_{n}/l)\in B)\geq(k_{n}/n)^{\vartheta}\}.\label{eq:l-def}
\end{equation}
Then $\ell_{n}^{-}(B)\leq\ell_{n}^{+}(B)$. Let $\ell_{n}(B)$ be
the result of an algorithm designed to satisfy
\begin{equation}
\ell_{n}(B)\in[\ell_{n}^{-}(B),\ell_{n}^{+}(B)];\label{eq:l_G_def}
\end{equation}
for the present analysis, it is sufficient to assume that $\ell_{n}(B)$
is a random variable satisfying (\ref{eq:l_G_def}). Now consider
the following generalisation of the estimator (\ref{eq:estimator-I})
for $\pi(B):=P(X\in B)$:
\begin{equation}
\hat{\pi}_{n}^{\mathrm{II}}(B):=\left(\hat{p}_{n}(\hat{Q}_{n}(\hat{Y}_{n}/\ell_{n}(B))\in B)\right)^{1/\ell_{n}(B)}.\label{eq:def_estimator-II}
\end{equation}

\begin{theorem}
\label{thm: estimator-II}For $X^{(1)},X^{(2)},...$, $(k_{n})$ and
$c'$ as in Theorem \ref{thm: estimator-I}, consider the estimator
(\ref{eq:def_estimator-II}) for $P(X\in B)$, with the quantile estimator
(\ref{eq:q_hat}) satisfying (\ref{eq:nu_conv-0}) and with $\xi\in(0,(1-c')^{-1})$
and $\vartheta\in(0,\xi]$. Then for $B_{n,\tau}$ as in Theorem \ref{thm: estimator-I},
\begin{equation}
\lim_{n\rightarrow\infty}\sup_{\tau\in[T^{-1},T]}\left|\frac{\log\hat{\pi}_{n}^{\mathrm{II}}(B_{n,\tau})}{\log P(X\in B_{n,\tau})}-1\right|=0\quad a.s.\quad T>1.\label{eq:claim_estimator-II}
\end{equation}

\end{theorem}
The proof can be found in Subsection \ref{sub:Proof-estimator-II}.

The constraints on $(k_{n})$, $\xi$ and $\vartheta$ ensure that
$n\hat{p}_{n}(\hat{Q}_{n}(\hat{Y}_{n}/\ell_{n}(B_{n,\tau}))\in B_{n,\tau})$
is eventually bounded by powers of $n$ with exponents in $(0,1)$.
This does not seem restrictive for applications. 

In practice, based on a few trial values of $\ell_{n}(B)$ which give
``acceptable'' numbers of $n\hat{p}_{n}(\hat{Q}_{n}(\hat{Y}_{n}/\ell_{n}(B))\in B)$,
one could check the stability of $\hat{\pi}_{n}^{\mathrm{II}}(B)$
with respect to $n\hat{p}_{n}(\hat{Q}_{n}(\hat{Y}_{n}/\ell_{n}(B))\in B)$.

\section{\label{sec:Examples}Numerical examples}

First, we will discuss simulations, considering the case of a bivariate
normal random vector $U$ with standard normal marginals and correlation
coefficient $\rho=0.5$. We are not yet concerned with marginal estimation,
so for $X$, we take the random vector with standard exponential marginals
obtained from $U$ by marginal transformations; therefore, $X=Y$
in this case. 

As extreme events, we will consider halfspaces, \emph{i.e.}, $U\in\{x\in\mathbb{R}^{2}:\: a_{1}x_{1}+a_{2}x_{2}>c\}$
for some $a\in\mathbb{R}^{2}$ and $c>0$; their probabilities are
easily calculated. In terms of $X$, these events are represented
by $X\in B$ with
\begin{equation}
B=\{x\in[0,\infty)^{m}:\: a_{1}\Phi^{-1}(1-\textrm{e}^{-x_{1}})+a_{2}\Phi^{-1}(1-\textrm{e}^{-x_{2}})>c\}.\label{eq:B_example}
\end{equation}

Experiments were performed with $a_{2}=1$ and with several different
values of $a_{1}$, with $c$ in each case chosen to ensure that $P(X\in B)=4\cdot10^{-8}$.
In all experiments, $n=5000$, and the estimator (\ref{eq:estimator-I})
was applied with $\xi=1$ and $k_{n}=20$. 

Our first case concerns $a_{1}=0.5$; $B$ is shown as a grey patch
in Figure \ref{fig:bivnormal_illustration}.1. Figure \ref{fig:bivnormal_illustration}.1(a)
shows $\hat{Y}_{n}$, which has no datapoint in $B$. The stretched
data cloud $\hat{Y}_{n}/\ell_{n}^{+}(B)$ is shown in Figure \ref{fig:bivnormal_illustration}.1(b),
with $k_{n}=20$ datapoints in $B$; $\ell_{n}^{+}(B)$ equals 0.334.
According to (\ref{eq:estimator-I}), the probability of $B$ is estimated
as $(20/5000)^{1/0.334}=6.6\cdot10^{-8}$. 

To appreciate how this estimator differs from the classical approach,
an estimator similar to (\ref{eq:estimator-I}) but based on the classical
multivariate tail limit (\ref{eq:conv_expmeasure}) was applied as
well: $\hat{\pi}_{n}^{\mathrm{cl}}(B):=(k_{n}/n)\textrm{e}^{-\lambda_{n}(B)}$
with $\lambda_{n}(B):=\inf\{l>0:\:\hat{p}_{n}(\hat{Y}_{n}+l)\in B)\geq k_{n}/n\}$.
It is similar to the estimator considered in \citet{deHaan=000026Sinha}
and \citet{Drees =000026 de Haan} without marginal estimation. Figure
\ref{fig:bivnormal_illustration}.1(c) shows $\hat{Y}_{n}+\lambda_{n}(B)$
with $\lambda_{n}(B)=8.92$; the corresponding probability estimate
equals $5.4\cdot10^{-7}$. The qualitative difference between Figures
\ref{fig:bivnormal_illustration}.1(b) and \ref{fig:bivnormal_illustration}.1(c)
is striking. 

Figure \ref{fig:bivnormal}.2 summarises the results of simulations
with $a_{1}=1,0.5,0.1,0,-0.1,-0.5$ and $a_{2}=1$, in each case with
500 realisations: (a) shows the boundaries of the events considered,
labelled by $a_{1}$; (b) shows the root mean square errors (RMSE)
of the logarithms of the probability estimates, and (c) shows the
bias of the logarithms of the probability estimates. For the bivariate
normal with $\rho<1$, the limiting measure in (\ref{eq:conv_expmeasure})
is concentrated on the boundaries, so the classical estimator is not
expected to do well in this case. Therefore, in addition, a correction
of the classical estimator based on residual tail dependence (\ref{eq:rtd})
was applied \emph{cf.} \citet{Draisma}. 

The results in Figure \ref{fig:bivnormal}.2 indicate that the standard
deviation is generally small in comparison to the bias, despite the
small value of $k_{n}$ used. The two classical estimators perform
better or worse depending on the value of $a_{1}$, but the LDP-based
estimator (\ref{eq:estimator-I}) performs consistently as good as
or better than both classical estimators in all cases. 

\begin{figure}
\label{fig:bivnormal_illustration}\includegraphics[scale=0.28]{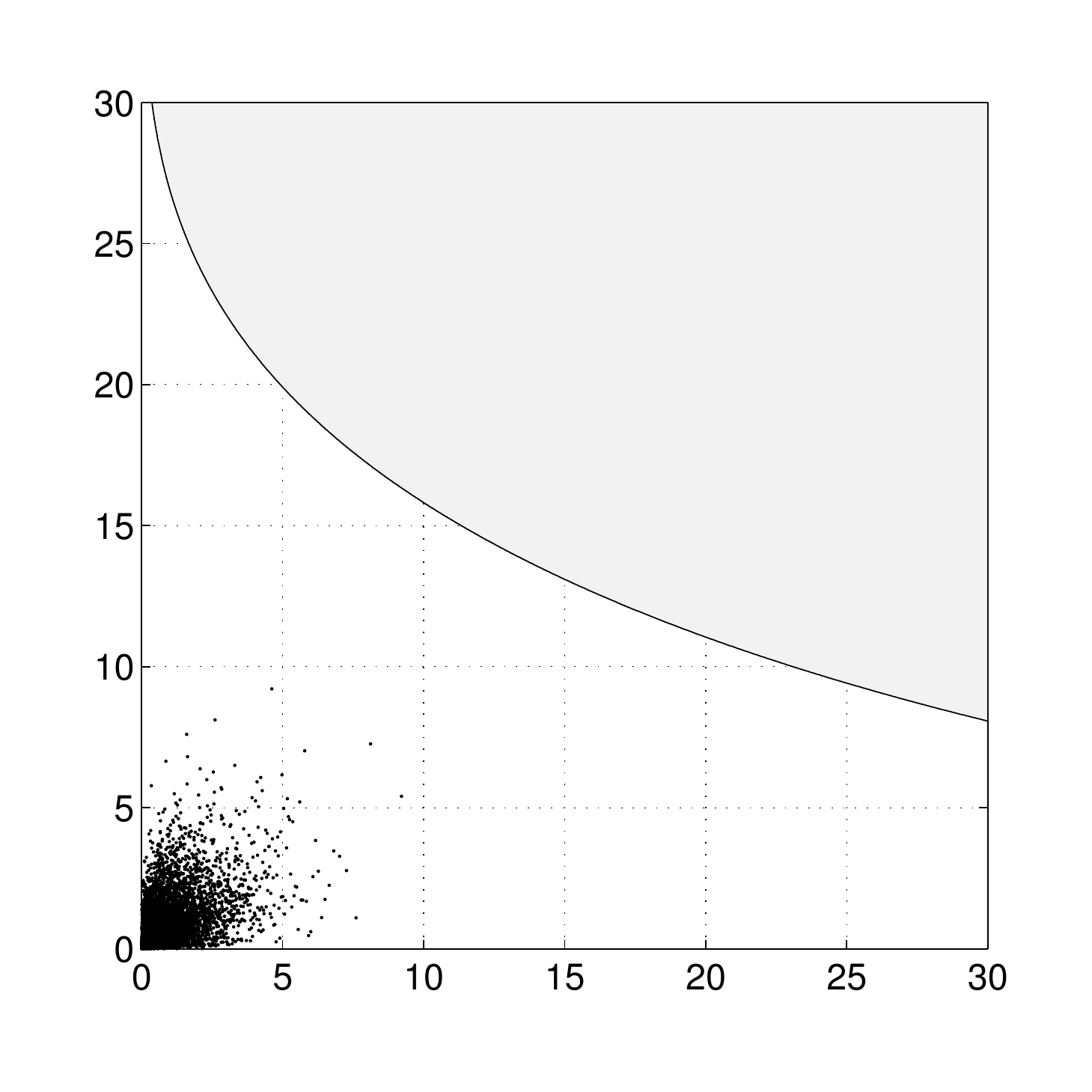}\includegraphics[scale=0.28]{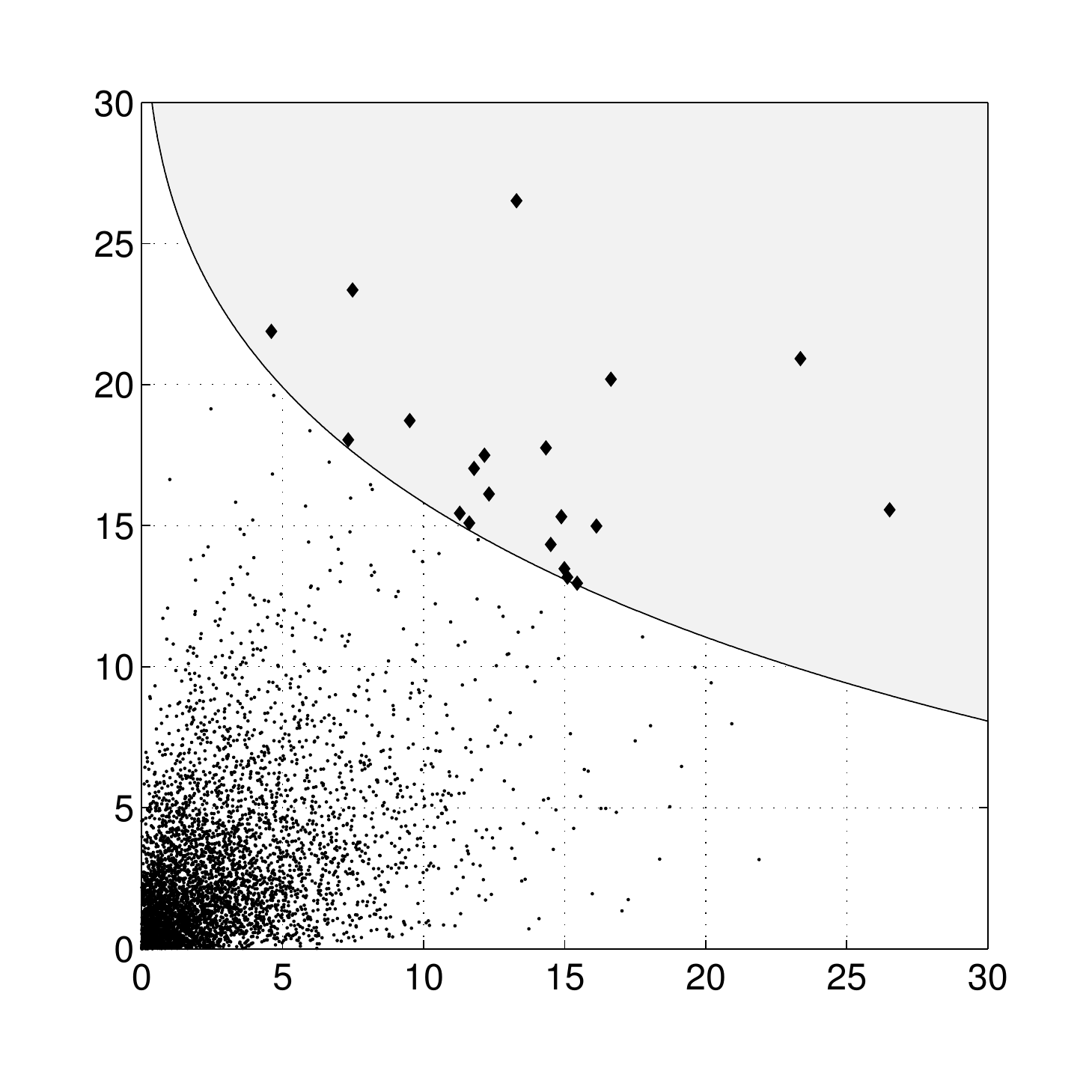}\includegraphics[scale=0.28]{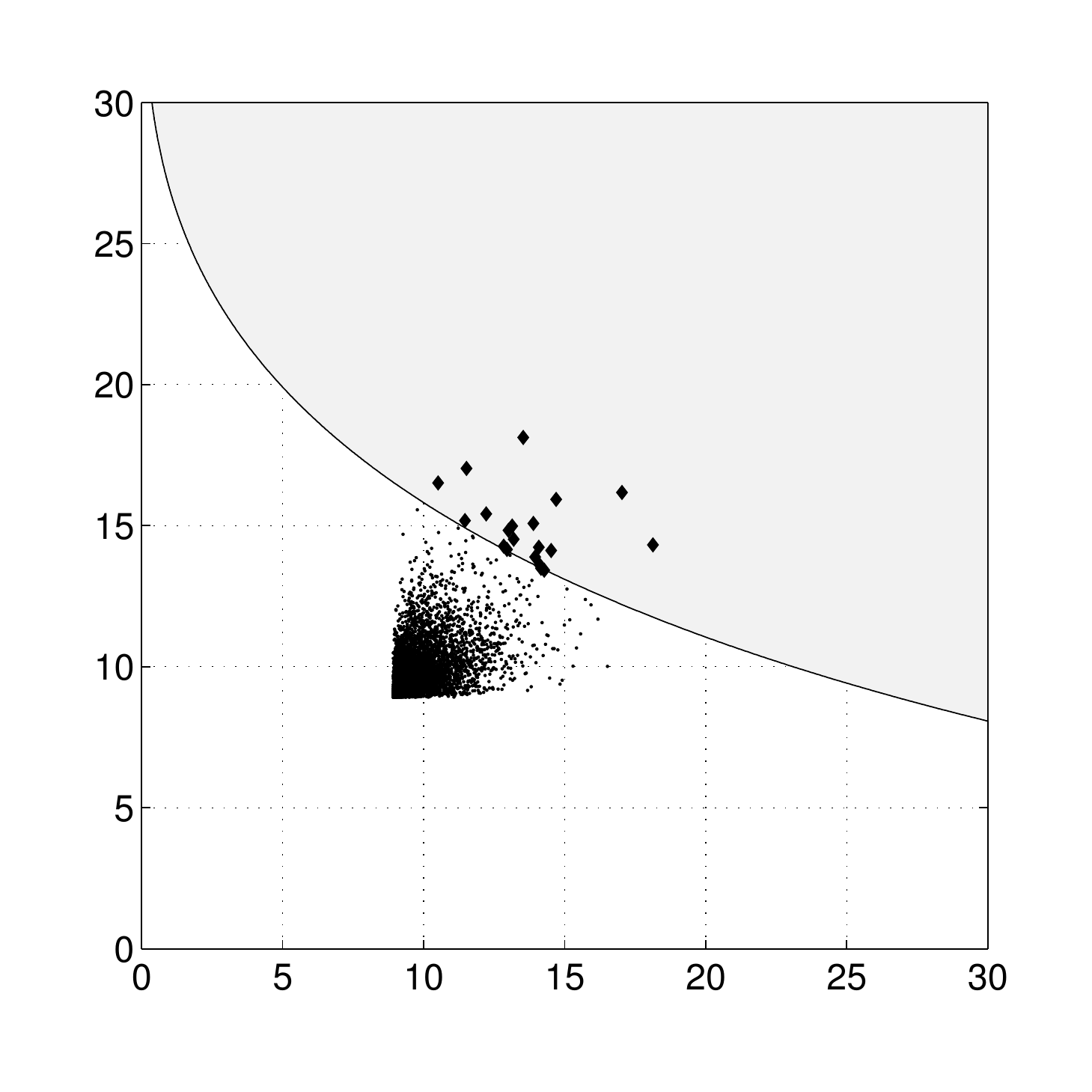}

\protect\caption{Simulation with bivariate normal dependence and exponential marginals
(see text). From left to right: (a) $\hat{Y}_{n}$ (dots) and failure
event $B$ given by (\ref{eq:B_example}) (grey); (b) $\hat{Y}_{n}/\ell_{n}^{+}(B)$
and failure event; (c) the classical analogue $\hat{Y}_{n}+\lambda_{n}(B)$
of (b) (see main text).}
\end{figure}

\begin{figure}
\label{fig:bivnormal}\includegraphics[scale=0.28]{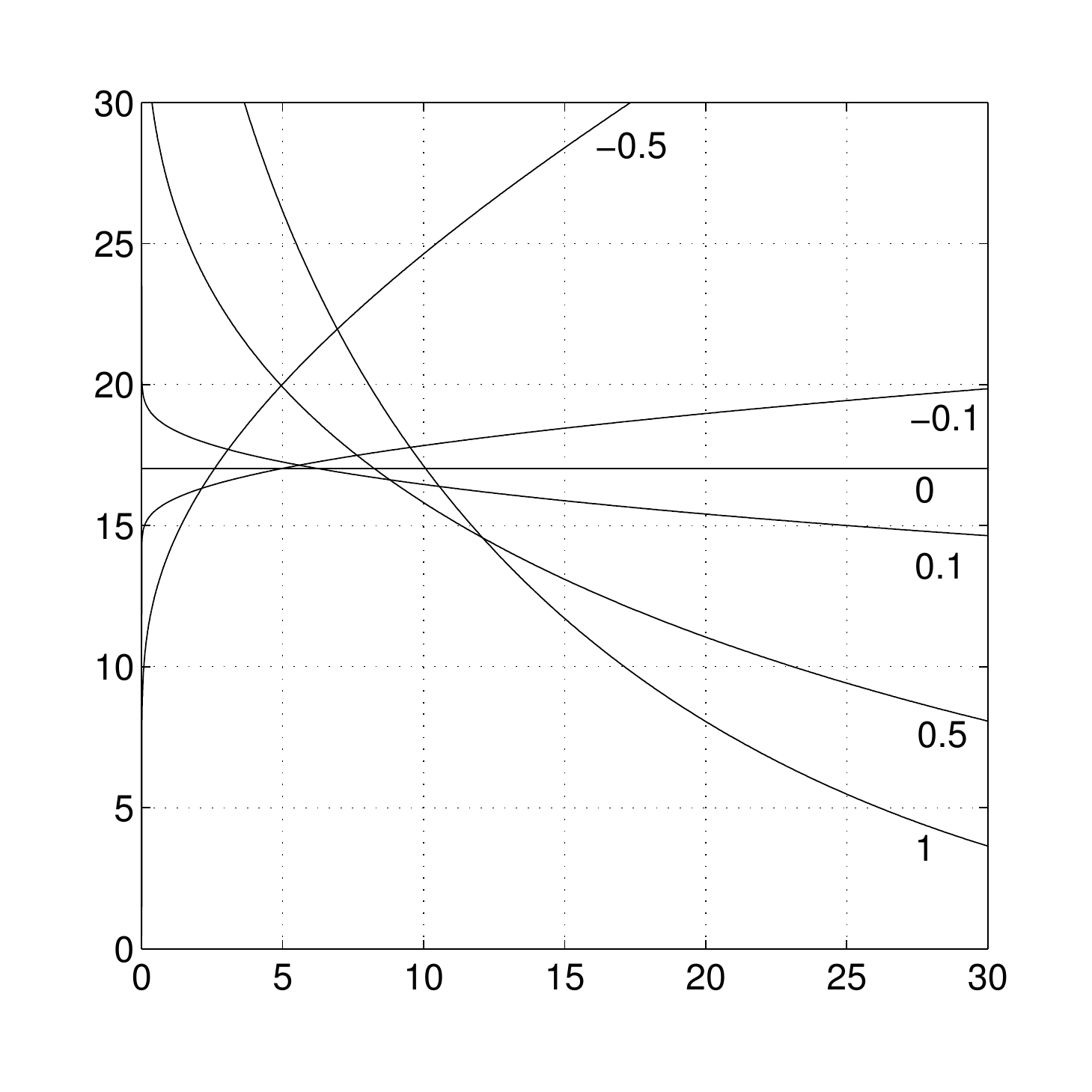}\includegraphics[scale=0.28]{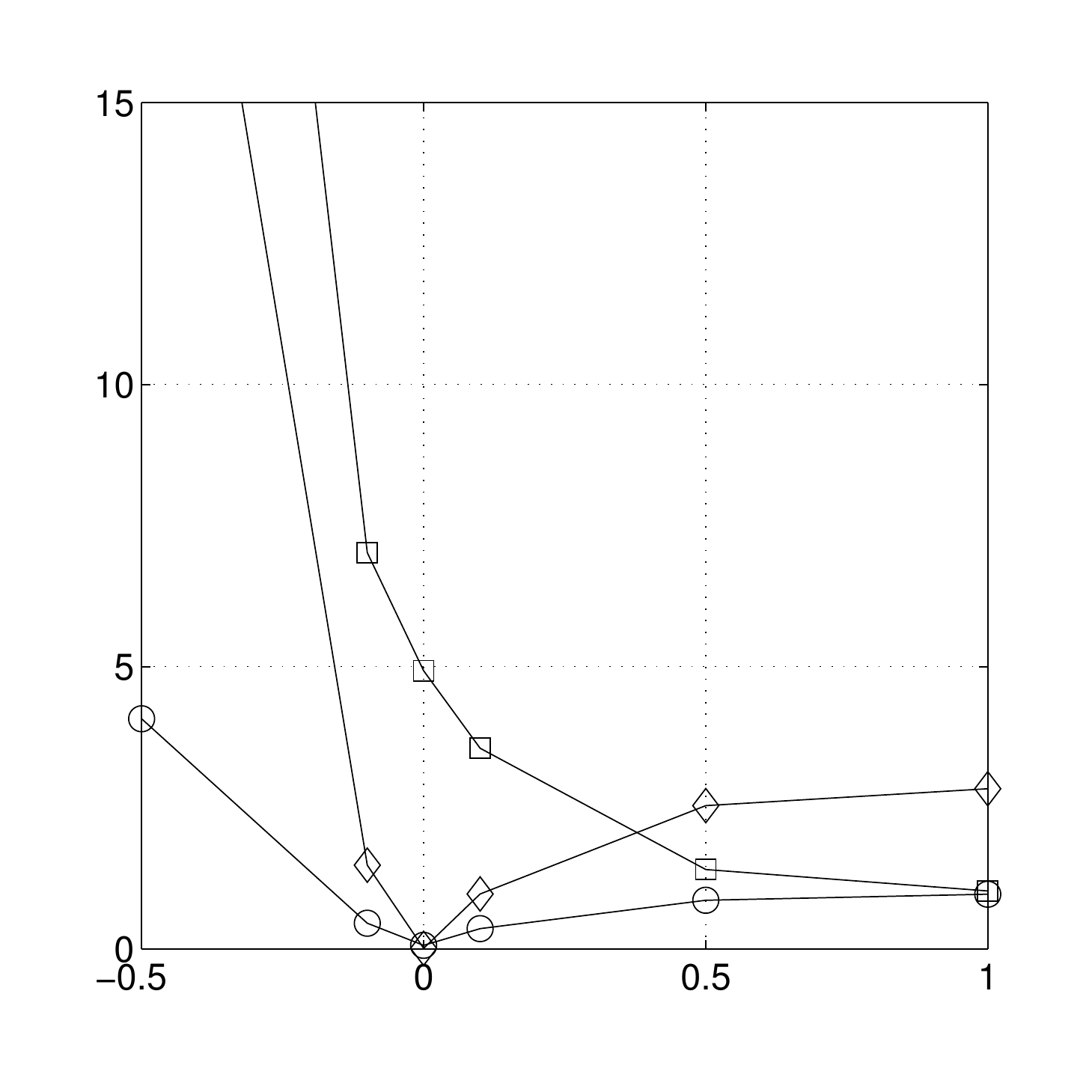}\includegraphics[scale=0.28]{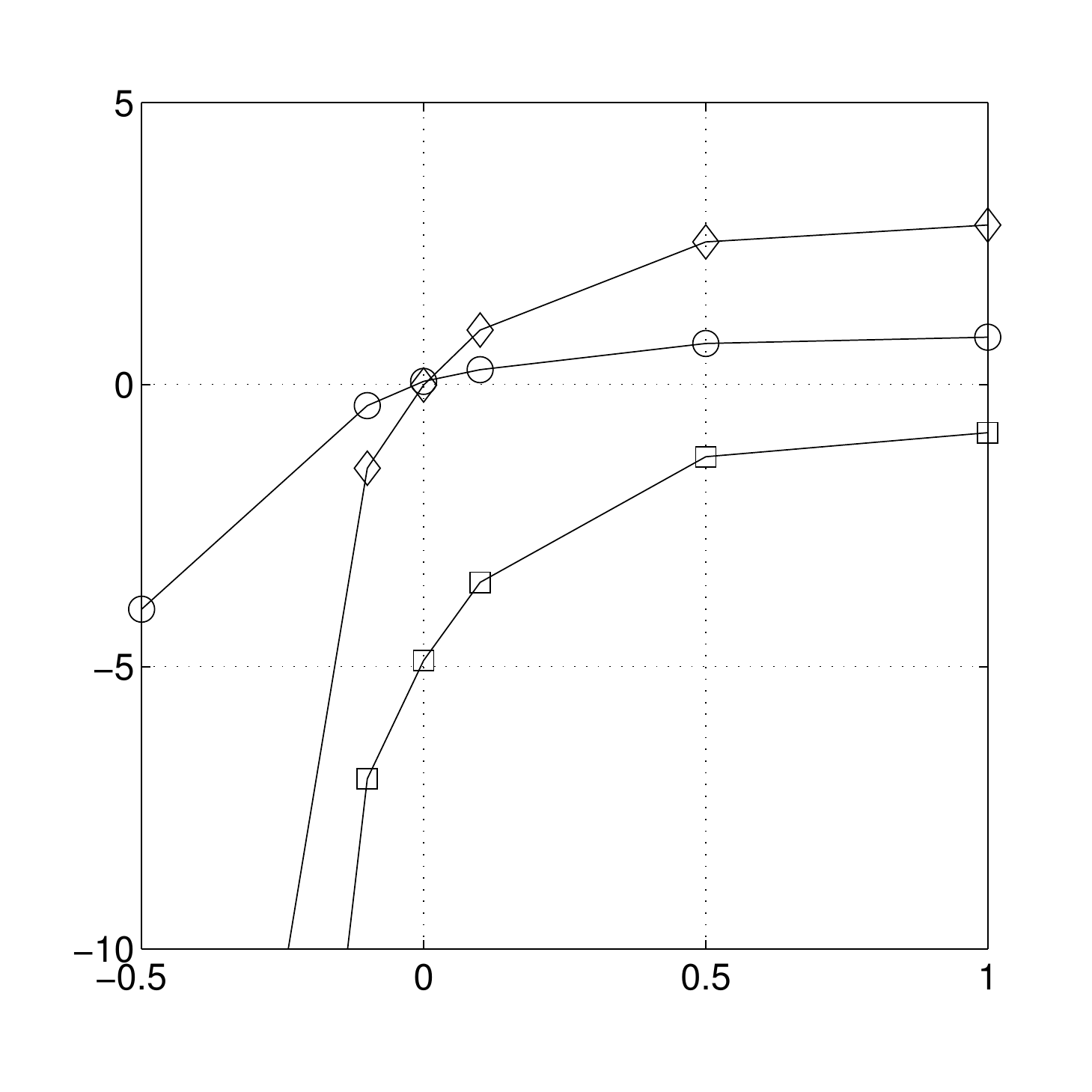}

\protect\caption{Simulations with bivariate normal dependence and exponential marginals
(see text). From left to right: (a) boundaries of failure events (\ref{eq:B_example})
labelled by $a_{1}$, for $a_{2}=1$; (b) RMSE of the logarithm of
probability as function of $a_{1}$ for estimator (\ref{eq:estimator-I})
(circles), its classical analogue (diamonds) and its classical analogue
accounting for residual tail dependence (squares); (c) bias for the
estimators as under (b). }
\end{figure}

The probability estimator (\ref{eq:estimator-I}) can also be applied
to estimate the survival function. This makes it possible to compare
it to the estimator for the survival function proposed in \citet{Wadsworth=000026Tawn}
(Sections 5.1 and 5.2) based on (\ref{eq:short-range-SF}) with $\inf I(A_{a})$
replaced by $\kappa(a)$, estimated using an approach employing the
Hill estimator. With both estimators, the same simulations were carried
out as reported in Section 5.3 of \citet{Wadsworth=000026Tawn}: for
$X$ considered above, estimates of the survival function $F^{c}(x_{1},x_{2})$
were made with $x_{2}=1.5\log n$ and $x_{1}/x_{2}=0.05,0.10,...,0.50$.
With $n=5000$, $k_{n}=20$, $500$ realisations, and $\hat{Y}_{n}$
replaced by the exact $Y=X$ as in \citet{Wadsworth=000026Tawn},
the RMSE of the logarithm of probability for (\ref{eq:estimator-I})
was 11-17\% higher than for the estimator from \citet{Wadsworth=000026Tawn},
which performed similarly to an estimator based on the conditional
probability approach of \citet{HeffernanTawn} (see Section 5.3 of
\citet{Wadsworth=000026Tawn}). This is an encouraging result for
an estimator as simple and widely applicable as (\ref{eq:estimator-I}). 

The final case is a trial application of the estimator (\ref{eq:def_estimator-II})
to an oceanographic dataset, in order to estimate the mean fraction
of time that wave run-up reaches the crest of a fictitious seawall.
Figure \ref{fig:YM6_analysis}.3 (upper right) shows simultaneous
3-hourly values of wave period measured at the offshore site YM6 in
the North Sea and surge level from the nearby harbour of IJmuiden
provided by Rijkswaterstaat, the Netherlands%
\footnote{Wave period is $T_{m_{-1,0}}$ (s); surge is still water level minus
estimated astronomical tide (m).%
}. The dataset covers 24 years ($n=70128$). For this trial, a strongly
simplified version of a model from \citet{TAW} is used to approximate
the run-up height of the 2\% highest waves on the seawall from wave
period and still water level. The set $B$ of wave period/water level
combinations leading to wave run-up exceeding 15m is indicated by
the grey area in the same figure. In the model, the mean depth at
the seawall base is 0m and the seawall has a flat smooth 1:4 slope.
The RMS wave height at the base is approximated by its upper bound
from \citet{Ruessink}. For the water level, we use surge data, ignoring
the astronomical tide. Dependence on wave direction is ignored in
marginals and nearshore wave transformation. Because of all these
simplifications, estimates obtained do not carry concrete relevance
to coastal flood safety.

Quantile estimates for wave period and surge were made using the simple
log-GW-based quantile estimator from \citet{De Valk I}: let $(k_{2,n})$
be a nondecreasing intermediate sequence satisfying that $\limsup_{n\rightarrow\infty}\log k_{2,n}/\log n<1$
and $\lim_{n\rightarrow\infty}k_{2,n}/\log_{2}n=\infty$ (with $\log_{2}$
the iterated logarithm), fix some $\iota>1$, and define
\begin{equation}
k_{i,n}:=\left\lfloor (n/k_{2,n})^{-\iota^{i-2}}n\right\rfloor \quad\textrm{for}\: i\in\{0,1\}.\label{eq:k_j}
\end{equation}

Taking $k_{n}=k_{0,n}$ in (\ref{eq:q_hat}), and
\begin{equation}
\hat{\theta}_{n,j}:=\frac{\log_{2}\frac{X_{j,n-k_{2,n}+1:n}}{X_{j,n-k_{1,n}+1:n}}-\log_{2}\frac{X_{j,n-k_{1,n}+1:n}}{X_{j,n-k_{0,n}+1:n}}}{\log\iota}\quad\textrm{and}\quad\hat{g}_{n,j}:=\frac{\log\frac{X_{j,n-k_{1,n}+1:n}}{X_{j,n-k_{0,n}+1:n}}}{h_{\hat{\theta}_{n,j}}(\iota)},\label{eq:params}
\end{equation}
(\ref{eq:nu_conv-0}) is ensured by Theorem 4 of \citet{De Valk I}.
Like the \citet{PickandsIII} estimator for the extreme value index,
the estimator $\hat{\theta}_{n}$ is based on only three order statistics.
However, its behaviour is entirely different, because the spacings
between $(k_{0,n})$, $(k_{1,n})$ and $(k_{2,n})$ are different. 

\begin{figure}
\label{fig:YM6_analysis}\includegraphics[scale=0.28]{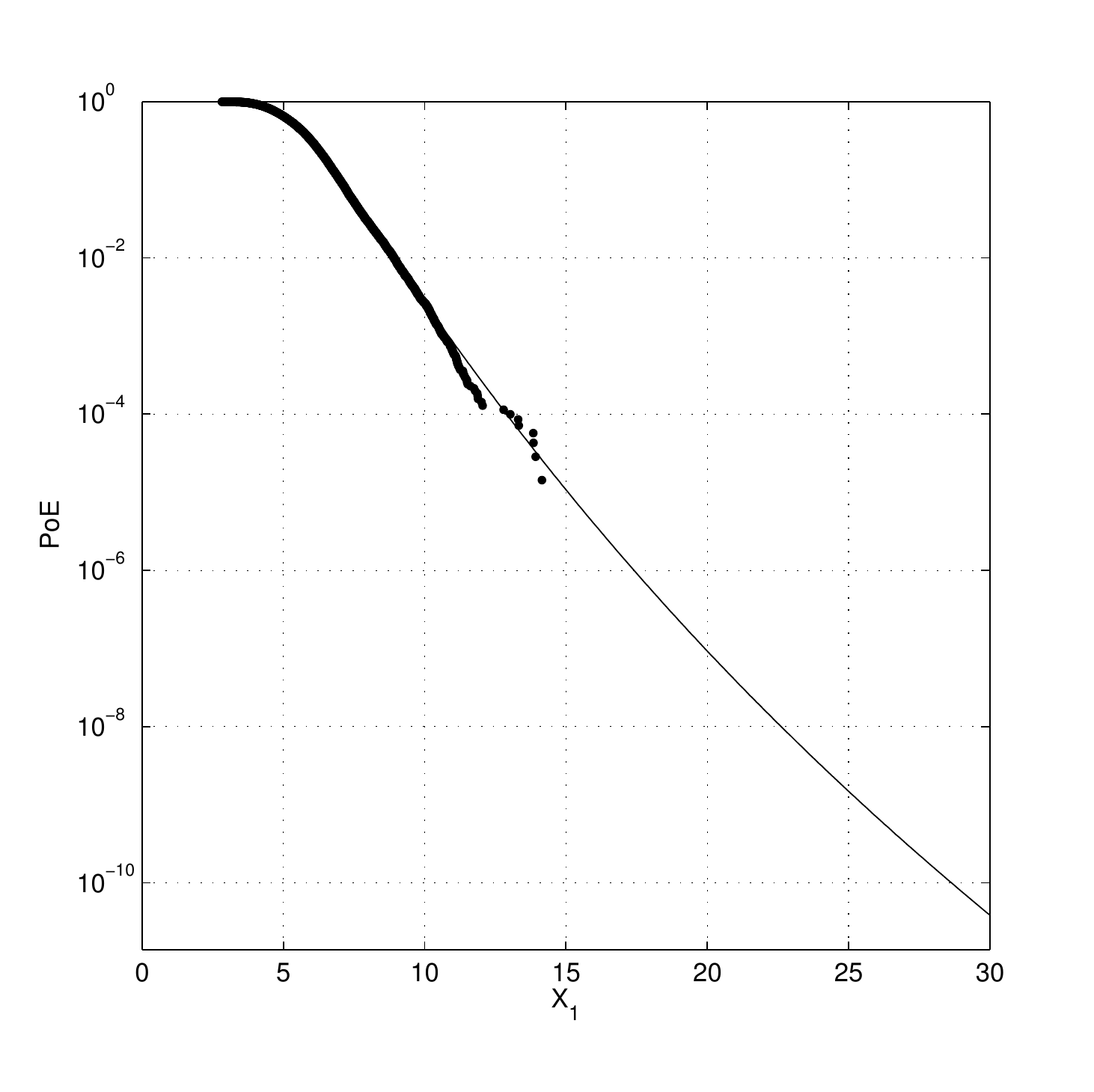}\includegraphics[scale=0.28]{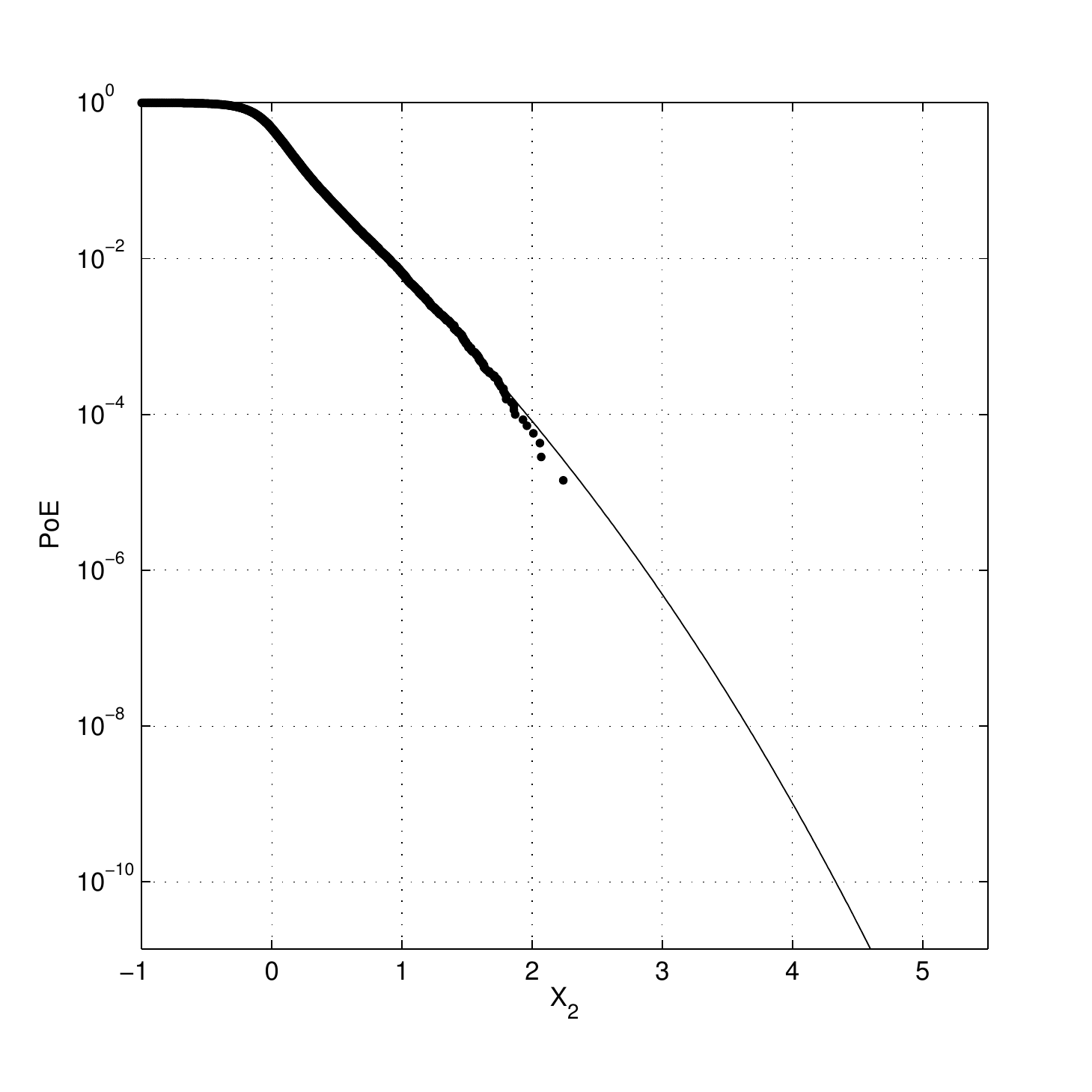}\includegraphics[scale=0.28]{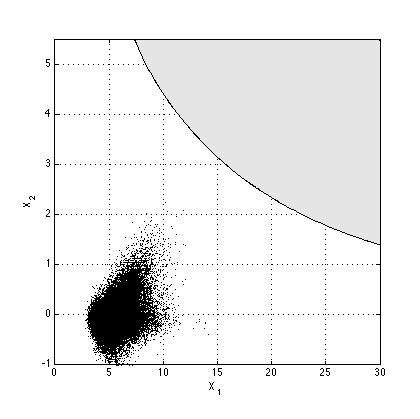}

\includegraphics[scale=0.28]{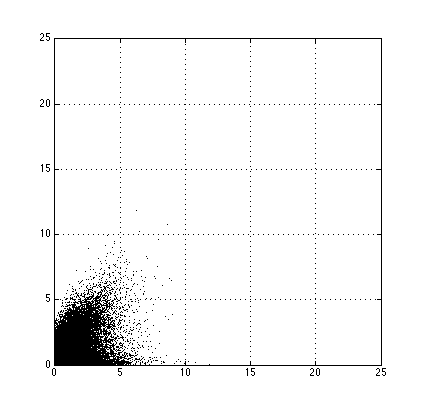}\includegraphics[scale=0.28]{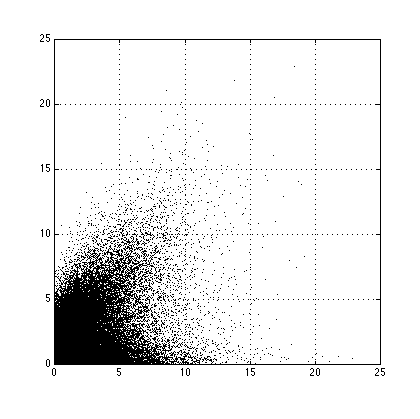}\includegraphics[scale=0.28]{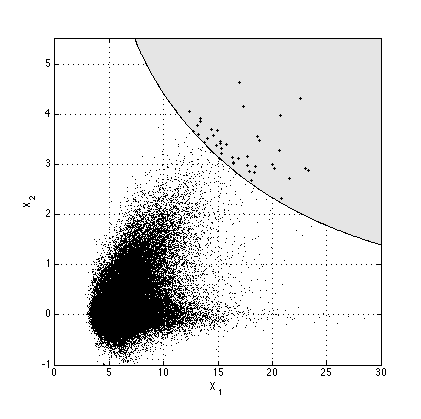}\protect\caption{Marginal log-GW tail estimates for wave period $X_{1}$ (upper left)
and surge level $X_{2}$ (upper middle); sample of $X$ and set $B$
corresponding to wave run-up exceeding 15m (upper right); $\hat{Y}_{n}$
(lower left); $\hat{Y}_{n}/\ell_{n}(B)$ (lower middle); $\hat{Q}_{n}(\hat{Y}_{n}/\ell_{n}(B))$
(lower right); fat dots indicating points with run-up exceeding 15m. }
\end{figure}
 Quantile estimates (\ref{eq:q_hat}) with (\ref{eq:k_j})-(\ref{eq:params}),
$\iota=2$ and $k_{n}=5009$ are shown in the upper left and middle
panels of Figure \ref{fig:YM6_analysis}.3. The lower left panel shows
$\hat{Y}_{n}$, the lower middle panel shows $\hat{Y}_{n}/\ell_{n}(B)$
with $\ell_{n}(B)=1/2.13$; $\hat{Q}_{n}(\hat{Y}_{n}/\ell_{n}(B))$
is shown in the lower right panel. In this case, $\hat{p}_{n}(\hat{Q}_{n}(\hat{Y}_{n}/\ell_{n}(B))\in B)=41/n=5.85\cdot10^{-4}$,
so for the estimator (\ref{eq:def_estimator-II}), $\hat{\pi}_{n}^{\mathrm{II}}(B)=(41/n)^{2.13}=1.3\cdot10^{-7}$,
about 11 hours per 10,000 years. 

Contrary to the assumptions made earlier, the 3-hourly surge and wave
period are serially dependent. Since we are estimating a fraction
of time, serial dependence does not need to invalidate the estimate;
its principal effect is that the estimate is less precise than it
would have been if the process were \emph{iid}. Imposing a minimum
separation of 24 hours between storm events, the 41 datapoints moved
into $B$ in Figure \ref{fig:YM6_analysis}.3 (lower right) represent
18 distinct events, giving a mean duration per event of 6.8 hours.
Using this value, the estimate $\hat{\pi}_{n}^{\mathrm{II}}(B)$ can
be converted to an estimate of the frequency of wave run-up exceeding
15m; its value is $1.7\cdot10^{-4}$ per year. Evidently, this unconventional,
but intuitively appealing variation of the peaks-over-threshold approach
would need formal underpinning by a model of serial dependence in
order to be taken seriously.

\section{\label{sec:Discussion}Discussion}

Like similar methods in the classical setting (\emph{e.g.} \citet{deHaan=000026Sinha,Drees =000026 de Haan,Draisma}),
the estimators (\ref{eq:estimator-I}) and (\ref{eq:def_estimator-II})
exploit homogeneity of a function describing tail dependence; in this
case, homogeneity (\ref{eq:I_mexp_symm}) of the rate function $I$.
This offers the advantage that no explicit estimate of $I$ is required.
However, in certain situations, there may be good reasons to estimate
$I$, such as if for a given random vector $X$, probabilities need
to be estimated for multiple sets in a consistent and reproducible
manner. Therefore, estimation of $I$ remains a topic deserving elaboration. 

The limitation of $A$ to continuity sets of $I$ in Theorems \ref{thm: estimator-I}
and \ref{thm: estimator-II} is less restrictive than it may seem,
since the homogeneity of $I$ makes continuity sets rather common,
as noted in Section \ref{sec:Dependence}. The other conditions on
$A$ are weak. 

To prove convergence of the estimators under such weak conditions,
local uniformity in $d$ of convergence in (\ref{eq:sampling_LDL_unif_A})
is employed, which is derived from uniformity in $d$ of convergence
in (\ref{eq:sample_LDL}). The latter also ensures local uniformity
in $\lambda$ of convergence in (\ref{eq:sampling_LDL_unif_A}), and
therefore local uniformity in $\tau$ of convergence of the estimators
in (\ref{eq:claim_estimator-I}) and (\ref{eq:claim_estimator-II}).
In practice, this means that if such an estimator applied to a given
dataset produces a fair estimate of $P(X\in B_{0})$ for some $B_{0}\subset\mathbb{R}^{m}$,
then it may also be applied with confidence to the same dataset to
estimate the probability of $B_{1}\subset\mathbb{R}^{m}$ such that
$P(X\in B_{1})\geq P(X\in B_{0})^{\tau}$ for $\tau>1$ not too large,
\emph{e.g.} $\tau=2$. If $P(X\in B_{0})\ll1$, \emph{e.g.} $P(X\in B_{0})=0.01$,
this amounts to extrapolation over several additional orders of magnitude
in probability. How far one can extrapolate in practice will depend
on the rates of convergence to the marginal log-GW tail limits and
in (\ref{eq:LDP_mexp}), which will differ from case to case. 

Convergence of log-probability ratios as in (\ref{eq:claim_estimator-I})
and (\ref{eq:claim_estimator-II}) is typical for the probability
range (\ref{eq:p_n}). A stronger notion of convergence might be desirable,
but would require restrictive additional assumptions which would be
hard to justify in applications. Rather, it is recommended to diagnose
bias in estimates and take this into account in estimates of uncertainty.
For this reason, modelling of bias and rate of convergence deserves
further study.

Deriving asymptotic error distributions will require additional assumptions
beyond those for Theorems \ref{thm: estimator-I} and \ref{thm: estimator-II}
and methods quite different from those employed in the present article.
Because it is complex (see \emph{e.g}. \citet{Drees =000026 de Haan}
for a comparable problem), this important topic needs to be left for
a follow-up study as well.

The theory is readily extended from events involving a high value
of at least one of the variables to events extreme ``in any direction'',
by replacing the exponential distribution as standard marginal by
the Laplace distribution \emph{cf.} \citet{Keef}. Other choices of
standard marginal are also possible, with minor adaptations to theory
and estimator. 

Furthermore, the main results of this article can be generalised straightforwardly
from a random vector in $\mathbb{R}^{m}$ to a random element of $\mathcal{C}_{b}(K)$,
the continuous functions on a compact metric space $K$. Classical
multivariate extreme value theory and estimation have been generalised
to this setting earlier; see \textit{e.g.} \citet{de Haan=000026Lin},
Part III of \citet{Laurens  boek}, \citet{EinmahlLin} and \citet{Ferreira=000026DeHaan}.
For the theory presented here, the main difference between the $\mathbb{R}^{m}$
setting and the $\mathcal{C}_{b}(K)$ setting is that in the latter,
exponential tightness of $\{P(Y/y\in\cdot),\: y>0\}$ no longer follows
from the exponential marginals; it is an independent assumption. In
loose terms, it entails that all but an exponentially small probability
mass is concentrated on equicontinuous sets of functions in $\mathcal{C}_{b}(K)$
(see \emph{e.g.} \citet{Dembo=000026Zeitouni}).

\section{\label{sec:lemmas}Proofs and lemmas}

\subsection{\label{sub:proof_ldp_complete}Proof of Theorem \ref{thm:ldp_complete}
and Corollary \ref{cor:ldp_complete}}

Convergence in (\ref{eq:logq-ERV}) is locally uniform in $\lambda$
(\textit{e.g.} \citet{Laurens  boek}, B.1.4 and B.2.9), so 
\begin{equation}
\lim_{y\rightarrow\infty}\sup_{\lambda\in[\Lambda^{-1},\Lambda]}\max_{j\in\{1,..,m\}}\left|h_{\theta j}^{-1}\left(\frac{\log q_{j}(y\lambda)-\log q_{j}(y)}{g_{j}(y)}\right)-\lambda\right|=0\quad\forall\Lambda>1.\label{eq:GW_K_unif}
\end{equation}

For every $y>0$, $\tilde{Q}_{y}$ is injective, so we can define
the random vector
\begin{equation}
\tilde{Y}_{y}:=\tilde{Q}_{y}^{-1}(X)=\tilde{Q}_{y}^{-1}Q(Y)\quad a.s.\label{eq:Ytilde_def}
\end{equation}
with $Y$ defined by (\ref{eq:Ydef}). By (\ref{eq:GW_K_unif}), there
exists almost surely for every $\Lambda>1$ and $\delta>0$ some $y_{\Lambda,\delta}>0$
such that for all $y\geq y_{\Lambda,\delta}$, $\left\Vert \tilde{Y}_{y}-Y\right\Vert _{\infty}>\delta y$
implies $\left\Vert Y\right\Vert _{\infty}>\Lambda y$. Therefore,
by (\ref{eq:sup_limit}), since $\Lambda>1$ is arbitrary, 
\begin{equation}
\limsup_{y\rightarrow\infty}\frac{1}{y}\log P(\left\Vert \tilde{Y}_{y}-Y\right\Vert _{\infty}>\delta y)\leq-\Lambda\quad\forall\varLambda>0.\label{eq:EE}
\end{equation}
By Theorem \ref{thm:ldp_exp}, the distribution functions of $\{Y/y,\: y>0\}$
satisfy the LDP (\ref{eq:LDP_mexp}) with good rate function $I$,
so (\ref{eq:EE}) implies the same for the distribution functions
of $\{\tilde{Y}_{y}/y,\: y>0\}$; see Theorem 4.2.13 of \citet{Dembo=000026Zeitouni}.
Therefore, (\ref{eq:LDP_mexp_approx}) follows from (\ref{eq:Ytilde_def}).
To prove (b), note that by (\ref{eq:LDP_mexp_approx}) and (\ref{eq:I_marg}),
\[
\lim_{y\rightarrow\infty}\frac{1}{y}\log\left(1-F_{j}\left(q_{j}(y)\textrm{e}^{g_{j}(y)h_{\theta_{j}}(\lambda)}\right)\right)=-\lambda\quad\forall\lambda\geq1
\]
for $j\in\{1,..,m\}$, so (\ref{eq:logq-ERV}) holds with $q=q_{j}$,
$g=g_{j}$ and $\theta=\theta_{j}$ for $j=1,..,m$. As in the proof
of (a), this implies (\ref{eq:EE}). Moreover, (\ref{eq:I_marg})
implies (\ref{eq:I_lb}), so $I$ is a good rate function. An application
of Theorem 4.2.13 of \citet{Dembo=000026Zeitouni} completes the proof
of the theorem. 

For the Corollary, note that for $A\subset[1,\infty)^{m}$, $P(X\in\tilde{Q}_{y}(yA))$
in (\ref{eq:LDP_mexp_approx}) is equal to
\[
P\left(\left(\frac{\log X_{1}-\log q_{1}(y)}{g_{1}(y)},..,\frac{\log X_{m}-\log q_{m}(y)}{g_{m}(y)}\right)\in H_{\theta}(A)\right)
\]
by (\ref{eq:qtilde-def}). Therefore, by the contraction principle
(see Theorem 4.2.1 in \citet{Dembo=000026Zeitouni}), (\ref{eq:LDP_mexp_tail})
follows from (\ref{eq:LDP_mexp_approx}).

\subsection{\label{sub:Proof-estimator-II}Proof of Theorem \ref{thm: estimator-II}}

For convenience, the following shorthand notation will be used:
\begin{equation}
\hat{\mu}_{n,l}(A):=\hat{p}_{n}(\hat{Q}_{n}(\hat{Y}_{n}/l)\in Q(y_{n}A)),\label{eq:mu_hat_def}
\end{equation}
and $l_{n}(A):=\ell_{n}(Q(y_{n}A))$, $l_{n}^{+}(A):=\ell_{n}^{+}(Q(y_{n}A))$,
$l_{n}^{-}(A):=\ell_{n}^{-}(Q(y_{n}A))$. 
\begin{svmultproof}
By Theorem \ref{thm:ldp_exp}, $Y$ defined by (\ref{eq:Ydef}) satisfies
the LDP (\ref{eq:LDP_mexp}) with good rate function $I$. As $\xi\in(0,(1-c')^{-1})$
with $c'$ as in (\ref{eq:k0_for_estimator}), take any $\Delta\in(\xi/\inf I(A),(1-c')^{-1}/\inf I(A))$.
Fixing an arbitrary $\Lambda>1$, then by Lemma \ref{lem: sampling_LDL},
for every $\delta\in(0,\Delta)$ (see (\ref{eq:mu_hat_def})),
\begin{equation}
\lim_{n\rightarrow\infty}\sup_{\lambda\in[\Lambda^{-1},\Lambda],\: d\in[\delta,\Delta]}\left|y_{n}^{-1}\log\hat{\mu}_{n,d/\lambda}(A\lambda)+d\inf I(A)\right|=0\quad a.s.\label{eq:sampling_LDL_unif_A}
\end{equation}
and
\begin{equation}
\limsup_{n\rightarrow\infty}\sup_{\lambda\in[\Lambda^{-1},\Lambda],\: d>\Delta}y_{n}^{-1}\log\hat{\mu}_{n,d/\lambda}(A\lambda)\leq-\Delta\inf I(A)<-\zeta\quad a.s.\label{eq:sampling_ubound_unif_A}
\end{equation}

Choosing $\delta<\vartheta/\inf I(A)$, since $\Delta>\xi/\inf I(A)$,
we observe that
\[
d\inf I(A)\begin{cases}
\leq\xi & \textrm{if}\: d\in[\delta,\xi/\inf I(A)]\subset[\delta,\Delta]\\
>\xi & \textrm{if}\: d\in(\xi/\inf I(A),\Delta]\subset[\delta,\Delta]
\end{cases}
\]
in (\ref{eq:sampling_LDL_unif_A}). Therefore, with (\ref{eq:sampling_ubound_unif_A}),
using (\ref{eq:l+def}), 
\begin{equation}
\lim_{n\rightarrow\infty}\sup_{\lambda\in[\Lambda^{-1},\Lambda]}\left|\lambda l_{n}^{+}(A\lambda)-\xi/\inf I(A)\right|=0\quad a.s.\label{eq:l+conv}
\end{equation}
and similarly, using (\ref{eq:l-def}), we find that
\begin{equation}
\lim_{n\rightarrow\infty}\sup_{\lambda\in[\Lambda^{-1},\Lambda]}\left|\lambda l_{n}^{-}(A\lambda)-\vartheta/\inf I(A)\right|=0\quad a.s.\label{eq:l-conv}
\end{equation}

By (\ref{eq:l+conv}), (\ref{eq:l-conv}), (\ref{eq:l_G_def}) and
(\ref{eq:sampling_LDL_unif_A}),
\begin{equation}
\lim_{n\rightarrow\infty}\sup_{\lambda\in[\Lambda^{-1},\Lambda]}\left|y_{n}^{-1}l_{n}^{-1}(A\lambda)\log\hat{\mu}_{n,l_{n}(A\lambda)}(A\lambda)+\lambda\inf I(A)\right|=0\quad a.s.\label{eq:l_conv}
\end{equation}
or equivalently, by (\ref{eq:def_estimator-II}),
\begin{equation}
\lim_{n\rightarrow\infty}\sup_{\lambda\in[\Lambda^{-1},\Lambda]}\left|y_{n}^{-1}\log\hat{\pi}_{n}^{\mathrm{II}}(Q(y_{n}A\lambda))+\lambda\inf I(A)\right|=0\quad a.s.\label{eq:Phat_conv}
\end{equation}

Since (\ref{eq:LDP_mexp_Q}) holds with $\inf I(A^{o})=\inf I(\bar{A})$
and $\inf I(A)>0$, by (\ref{eq:I_mexp_symm}) and (\ref{eq:Phat_conv}),
\[
\lim_{n\rightarrow\infty}\sup_{\lambda\in[\Lambda^{-1},\Lambda]}\left|\frac{\log\hat{\pi}_{n}^{\mathrm{II}}(Q(y_{n}A\lambda))}{\log P(X\in Q(y_{n}A\lambda))}-1\right|=0\quad a.s.,
\]
and (\ref{eq:claim_estimator-II}) follows from (\ref{eq:k0_for_estimator}),
since $\Lambda>1$ is arbitrary. 
\end{svmultproof}

\subsection{\label{sub:Proof-P-estimator-1}Proof of Theorem \ref{thm: estimator-I}}

Following the proof of Theorem \ref{thm: estimator-II} in Subsection
\ref{sub:Proof-estimator-II}, (\ref{eq:l+conv}) and (\ref{eq:estimator-I})
yield
\[
\lim_{n\rightarrow\infty}\sup_{\lambda\in[\Lambda^{-1},\Lambda]}\left|y_{n}^{-1}\log\hat{\pi}_{n}^{\mathrm{I}}(Q(y_{n}A\lambda))+\lambda\inf I(A)\right|=0
\]
and the result (\ref{eq:claim_estimator-I}) follows as in the proof
of Theorem \ref{thm: estimator-II}.

\subsection{Lemmas}
\begin{lemma}
\label{lem: for En}Let $Y$ be a random vector in $[0,\infty)^{m}$
with standard exponential marginals satisfying the LDP (\ref{eq:LDP_mexp})
with good rate function $I$, and $Y^{(1)},Y^{(2)},...$ a sequence
of iid copies of $Y$. Let the Borel set $A\subset[0,\infty)^{m}$
be a continuity set of $I$ satisfying $\inf I(A)\in(0,\infty)$.
If $(y_{n}>0)$ and $\Delta>0$ satisfy $\lim_{n\rightarrow\infty}y_{n}=\infty$
and\textup{
\begin{equation}
\Delta<\liminf_{n\rightarrow\infty}\frac{\log n}{y_{n}\inf I(A)}<\infty,\label{eq:cond_for_En}
\end{equation}
}then with $\hat{p}_{n}$ defined by (\ref{eq:p_hat_def}),
\begin{equation}
\lim_{n\rightarrow\infty}\sup_{d\in[\delta,\Delta]}\bigl|y_{n}^{-1}\log\hat{p}_{n}(Y\in dAy_{n})+d\inf I(A)\bigr|\rightarrow0\quad a.s.\quad\forall\delta\in(0,\varDelta)\label{eq:sample_LDL}
\end{equation}
and
\begin{equation}
\limsup_{n\rightarrow\infty}\sup_{d>\Delta}y_{n}^{-1}\log\hat{p}_{n}(Y\in dAy_{n})\leq-\Delta\inf I(A)\quad a.s.\label{eq:sample_ubound}
\end{equation}
\end{lemma}
\begin{svmultproof}
Let $\mathcal{A}:=\cup_{\lambda\geq1}(\lambda A)$; by (\ref{eq:I_mexp_symm}),
$\mathcal{A}$ is a continuity set of $I$ satisfying $\inf I(\mathcal{A})=\inf I(A)<\infty$.
Define the random variable
\begin{equation}
v:=\inf\{w>0:\: Yw\in\mathcal{A}\}\label{eq:omega_def}
\end{equation}
with $\inf\{\emptyset\}:=\infty$, and let $G$ be its distribution
function. Since $\cup_{\lambda\geq1}(\mathcal{A}\lambda)\subset\mathcal{A}$,
$Y\in\mathcal{A}^{o}z$ $\Rightarrow$ $v\leq z^{-1}$ $\Rightarrow$
$Y\in\mathcal{\bar{A}}z$ for every $z>0$, so by (\ref{eq:LDP_mexp})
and (\ref{eq:I_mexp_symm}),
\begin{equation}
\lim_{y\rightarrow\infty}y^{-1}\log G(w/y)=-w^{-1}\inf I(\mathcal{A})\quad\forall w>0.\label{eq:w_lim}
\end{equation}

Therefore, since $\inf I(\mathcal{A})\in(0,\infty)$, $-\log G(1/\textrm{Id})\in RV_{\{1\}}$,
so by \citet{Bingham} (Theorem 1.5.2) and (\ref{eq:w_lim}) again,
for every $a>0$,
\begin{equation}
\lim_{y\rightarrow\infty}\sup_{w\geq a}\bigl|y^{-1}\log G(w/y)+w^{-1}\inf I(\mathcal{A})\bigr|=0.\label{eq:lambda_lim}
\end{equation}

By (\ref{eq:w_lim}), there is for every $\varepsilon>0$ an $n_{\varepsilon}\in\mathbb{N}$
such that for all $n\geq n_{\varepsilon}$, 
\begin{equation}
nG(a/y_{n})\geq\textrm{e}^{\log n-(\varepsilon+a^{-1}\inf I(\mathcal{A}))y_{n}}\label{eq:nF}
\end{equation}

Taking $a=1/\Delta$, then by (\ref{eq:cond_for_En}), $\varepsilon>0$
can be chosen small enough that the exponent in (\ref{eq:nF}) eventually
exceeds $\varepsilon\log n$. Therefore,
\begin{equation}
\lim_{n\rightarrow\infty}nG(a/y_{n})/\log n=\infty.\label{eq:Flim}
\end{equation}

With $G^{-1}$ the left-continuous inverse of $G$, almost surely
$v^{(i)}=G^{-1}(\mathcal{U}^{(i)})$ for all $i\in\mathbb{N}$, with
$\mathcal{U}^{(1)},\mathcal{U}^{(2)},...$ independent and uniformly
distributed on $(0,1)$, so almost surely (see def. (\ref{eq:p_hat_def})),
$\hat{p}_{n}(v\leq w/y_{n})=\hat{p}_{n}(\mathcal{U}\leq G(w/y_{n}))$
for all $n\in\mathbb{N}$ and all $w\geq a$. Therefore, by \citet{Wellner}
(Corollary 1) and (\ref{eq:Flim}),
\begin{equation}
\lim_{n\rightarrow\infty}\sup_{w\geq a}\bigl|\log\hat{p}_{n}(v\leq w/y_{n})-\log G(w/y_{n})\bigr|=0\quad a.s.\label{eq:Ulim}
\end{equation}
and since $v\leq w/y_{n}\Rightarrow$$Y\in\mathcal{A}y_{n}/(wl)$
$\Rightarrow$ $v\leq wl/y_{n}$ for all $l>1$ and $w>0$, using
(\ref{eq:lambda_lim}) and (\ref{eq:I_mexp_symm}), as $a=1/\Delta$,
\begin{equation}
\lim_{n\rightarrow\infty}\sup_{d\in(0,\Delta]}\bigl|y_{n}^{-1}\log\hat{p}_{n}(Y\in d\mathcal{A}y_{n})+d\inf I(\mathcal{A})\bigr|=0\quad a.s.\label{eq:sample_LDL-1}
\end{equation}

Therefore, as $A\subset\mathcal{A}$ and $\inf I(\mathcal{A})=\inf I(A)$,
\begin{equation}
\limsup_{n\rightarrow\infty}\sup_{d\in(0,\Delta]}y_{n}^{-1}\log\hat{p}_{n}(Y\in dAy_{n})+d\inf I(A)\leq0\quad a.s.\label{eq:sample_A_ineq_up}
\end{equation}

$A$ is a continuity set of $I$ and $I$ satisfies (\ref{eq:I_mexp_symm}),
so there is for every $\varepsilon>0$ a point $x_{\varepsilon}\in A^{o}$
such that $I(x_{\varepsilon})<\inf I(A)+\varepsilon$. Let $\varepsilon>0$
and $\eta>1$ be such that $\Delta\eta(\inf I(A)+\varepsilon)<\liminf_{n\rightarrow\infty}y_{n}^{-1}\log n$
(see (\ref{eq:cond_for_En})). Then for $\eta$ sufficiently close
to $1$, an open set $\mathcal{B}\subset[0,\infty)^{m}$ can be constructed
such that
\[
\cup_{\lambda\geq1}(\lambda\mathcal{B})\subset\mathcal{B},\quad x_{\varepsilon}\in\mathcal{B}\setminus(\mathcal{B}\eta)\subset A^{o},\quad\textrm{and}
\]
\begin{equation}
\inf I(\mathcal{B}^{o})=\inf I(\bar{\mathcal{B}})\in(\inf I(A),I(x_{\varepsilon})]\label{eq:cond_B_i}
\end{equation}
as follows. The first two requirements on $\mathcal{B}$ are satisfied
by $\mathcal{B}'=\cup_{\lambda\geq1}(\lambda U)$ for some sufficiently
small neighbourhood $U\subset A^{o}$ of $x_{\varepsilon}$, with
$\eta>1$ close enough to $1$. If $\mathcal{B}'$ is a continuity
set of $I$, then set $\mathcal{B}=\mathcal{B}'$. Else, consider
the function $f:[0,\infty)^{m}\times[0,1]\rightarrow[0,\infty)^{m}$
defined by $f(y,a):=ay+(1-a)(\left\Vert y\right\Vert _{\infty}/\left\Vert x_{\varepsilon}\right\Vert _{\infty})x_{\varepsilon}$.
It satisfies $f(\mathcal{B}',1)=\mathcal{B}'$, $f(\mathcal{B}',0)=\mathcal{B}'\cap\cup_{\lambda>0}(\lambda x_{\varepsilon})$,
and $f(\mathcal{B}',a)\subset f(\mathcal{B}',a')$ if $a\leq a'$.
Therefore, $a\mapsto\inf I(f(\mathcal{B}',a))$ is nonincreasing,
so with $\alpha$ any of its continuity points in $(0,1)$, $\mathcal{B}=f(\mathcal{B}',\alpha)$
is a continuity set of $I$ and satisfies (\ref{eq:cond_B_i}). By
(\ref{eq:cond_B_i}),
\[
\hat{p}_{n}(Y\in dAy_{n})\geq\hat{p}_{n}(Y\in d\mathcal{B}y_{n})-\hat{p}_{n}(Y\in d\eta\mathcal{B}y_{n})
\]
\begin{equation}
=\hat{p}_{n}(Y\in d\mathcal{B}y_{n})(1-\textrm{e}^{\log\hat{p}_{n}(Y\in d\eta\mathcal{B}y_{n})-\log\hat{p}_{n}(Y\in d\mathcal{B}y_{n})})\label{eq:p_hat_B}
\end{equation}
and furthermore, (\ref{eq:sample_LDL-1}) continues to hold after
substituting $\mathcal{B}$ or $\mathcal{B}\eta$ for $\mathcal{A}$.
Therefore, by (\ref{eq:I_mexp_symm}), for every $\delta\in(0,\varDelta)$
almost surely, the right-hand side of (\ref{eq:p_hat_B}) is $\hat{p}_{n}(Y\in d\mathcal{B}y_{n})(1+o(1))$
uniformly in $d\in[\delta,\Delta]$ and furthermore, using (\ref{eq:cond_B_i}),
\begin{equation}
\liminf_{n\rightarrow\infty}\inf_{d\in[\delta,\Delta]}y_{n}^{-1}\log\hat{p}_{n}(Y\in dAy_{n})+dI(x_{\varepsilon})\geq0\quad a.s.\label{eq:sample_A_ineq_low}
\end{equation}
Now (\ref{eq:sample_LDL}) follows from (\ref{eq:sample_A_ineq_up})
and (\ref{eq:sample_A_ineq_low}), because $I(x_{\varepsilon})<\inf I(A)+\varepsilon$,
and $\varepsilon>0$ can be chosen arbitrarily close to $0$. Finally,
by (\ref{eq:sample_LDL-1}), as $\cup_{\lambda\geq1}(\mathcal{A}\lambda)\subset\mathcal{A}$,
\begin{equation}
\limsup_{n\rightarrow\infty}\sup_{d>\Delta}y_{n}^{-1}\log\hat{p}_{n}(Y\in d\mathcal{A}y_{n})\leq-\Delta\inf I(\mathcal{A})\quad a.s.\label{eq:sample_uboud-1}
\end{equation}
and because $A\subset\mathcal{A}$ and $\inf I(\mathcal{A})=\inf I(A)$,
(\ref{eq:sample_ubound}) follows.\end{svmultproof}

\begin{lemma}
\label{lem: for Bn}

Let $Y$ be a random vector on $[0,\infty)^{m}$ with standard exponential
marginals and $Y^{(1)},Y^{(2)},...$ a sequence of iid copies of $Y$.
Define $\hat{Y}_{n}^{(i)}:=(\hat{Y}_{1,n}^{(i)},..,\hat{Y}_{m,n}^{(i)})$
for $i=1,..,n$ with
\begin{equation}
\hat{Y}_{j,n}^{(i)}:=-\log(1-(R_{j,n}^{(i)}-{\scriptstyle \frac{1}{2}})/n).\label{eq:eq:Zhat_def}
\end{equation}
and $R_{j,n}^{(i)}:=\sum_{l=1}^{n}\mathbf{1}(Y_{j}^{(l)}\leq Y_{j}^{(i)})$.
For $(y_{n}>0)$ satisfying $\liminf_{n\rightarrow\infty}y_{n}/\log n>0$,
\begin{equation}
\sup_{\varepsilon>0}\limsup_{n\rightarrow\infty}y_{n}^{-1}\log\hat{p}_{n}\left(\bigl\Vert\hat{Y}_{n}-Y\bigr\Vert_{\infty}>y_{n}\varepsilon\right)=-\infty\quad a.s.\label{eq:forBn_claim}
\end{equation}
\end{lemma}
\begin{svmultproof}
Since  $\hat{p}_{n}(\bigl\Vert\hat{Y}_{n}-Y\bigr\Vert{}_{\infty}>y_{n}\varepsilon)\leq\sum_{j=1}^{m}\hat{p}_{n}(|\hat{Y}_{j,n}-Y_{j}|>y_{n}\varepsilon)$,
it is sufficient to prove (\ref{eq:forBn_claim}) for the univariate
case. 

Let $\varGamma_{n}$ and $\varGamma_{n}^{-1}$ be the empirical distribution
function and quantile function of $\mathcal{U}^{(1)},..,\mathcal{U}^{(n)}$,
with $\mathcal{U}^{(i)}:=\exp(-Y^{(i)})$ uniformly distributed in
$(0,1)$ for every $i\in\mathbb{N}$. Because $\sup_{t\in[1/n,1]}|t/\varGamma_{n}^{-1}(t)|=\sup_{t\in[0,1]}|t^{-1}\varGamma_{n}(t)|$
(see \emph{e.g.} \citet{Wellner}), by Theorem 2 of \citet{ShorackWellner1978},
\begin{equation}
\sup_{t\in[1/n,1]}\log(t/\varGamma_{n}^{-1}(t))/\log_{2}n\rightarrow1\quad a.s.\label{eq:SWupp}
\end{equation}

Similarly, because $\sup_{t\in[1/n,1]}|t^{-1}\varGamma_{n}^{-1}(t)|\vee1=\sup_{t\in[\mathcal{U}_{1:n},1]}|t/\varGamma_{n}(t)|$,
by Theorem 3 of \citet{ShorackWellner1978},
\[
\sup_{t\in[1/n,1]}(t^{-1}\varGamma_{n}^{-1}(t))/\log_{2}n\rightarrow1\quad a.s.
\]
so
\begin{equation}
\inf_{t\in[1/n,1]}\log(t/\varGamma_{n}^{-1}(t))/\log_{2}n\rightarrow0\quad a.s.\label{eq:SWlow}
\end{equation}

Since $Y_{n-i+1:n}=-\log\varGamma_{n}^{-1}(i/n)$ for $i=1,..n$ and
$y_{n}/\log_{2}n\rightarrow\infty$, (\ref{eq:SWupp}) and (\ref{eq:SWlow})
imply, using (\ref{eq:eq:Zhat_def}),
\begin{equation}
\max_{i\in\{1,..,n\}}|Y_{n-i+1:n}-\hat{Y}_{n-i+1:n}|/y_{n}\rightarrow0\quad a.s.\label{eq:YminYhat}
\end{equation}

As a consequence, there is almost surely for every $\delta>0$ an
$n_{\delta}\in\mathbb{N}$ such that for all $\varepsilon\geq\delta$,
$\hat{p}_{n}(|Y-\hat{Y}_{n}|>y_{n}\varepsilon)=0$ for all $n\geq n_{\delta}$
and therefore, that $y_{n}^{-1}\log\hat{p}_{n}(|Y-\hat{Y}_{n}|>y_{n}\varepsilon)=-\infty$,
proving the univariate case of (\ref{eq:forBn_claim}). \end{svmultproof}

\begin{lemma}
\label{lem: For Cn}Let $X$ be a random vector on $\mathbb{R}^{m}$
having continuous marginals satisfying log-GW tail limits, and let
$X^{(1)},X^{(2)},...$ be a sequence of iid copies of $X$. With $Q$,
$\hat{Q}_{n}$ and $\hat{Y}_{n}$ defined by (\ref{eq:Qdef}), (\ref{eq:Q_hat})
and (\ref{eq:Zhat_def_curly}), let $(k_{n})$ satisfy (\ref{eq:k0_for_estimator})
and $\hat{q}_{j,n}$ defined by (\ref{eq:q_hat}) satisfy (\ref{eq:nu_conv-0})
for $j=1,..,m$, with $y_{n}$ defined by (\ref{eq:def_y_n}). Then
for every $\delta>0$ and $\varepsilon>0$, 
\begin{equation}
\lim_{n\rightarrow\infty}\sup_{l\geq\delta}y_{n}^{-1}\log\hat{p}_{n}\left(\bigl\Vert Q^{-1}\hat{Q}_{n}(\hat{Y}_{n}l^{-1})-\hat{Y}_{n}l^{-1}\bigr\Vert_{\infty}>y_{n}\varepsilon\right)=-\infty\quad a.s.\label{eq:marg_estimator_conv}
\end{equation}
\end{lemma}
\begin{svmultproof}
Fix $\varepsilon>0$ and $\delta>0$. As in Lemma \ref{lem: for Bn},
we only need to prove (\ref{eq:marg_estimator_conv}) for the univariate
case, so we proceed with this. Note that (\ref{eq:marg_estimator_conv})
holds if an $n_{\delta,\varepsilon}\in\mathbb{N}$ exists such that
(suppressing the labels of vector components in the univariate case)
\begin{equation}
\sup_{l\geq\delta}\sup_{j\in\{1,..,n\}}|q^{-1}\hat{q}_{n}(\hat{Y}_{j:n}l^{-1})-\hat{Y}_{j:n}l^{-1}|\leq y_{n}\varepsilon\quad\forall n\geq n_{\delta,\varepsilon}\label{eq:ubound_QinvQhat-Id}
\end{equation}

Fixing $\Lambda>\max(1,\delta^{-1})/(1-c)\geq\max(1,\delta^{-1})\limsup_{n\rightarrow\infty}\log(2n)/y_{n}$
with $c$ as in (\ref{eq:k0_for_estimator}), (\ref{eq:ubound_QinvQhat-Id})
holds for some $n_{\delta,\varepsilon}\in\mathbb{N}$ if
\begin{equation}
\sup_{z\in[0,y_{n}\Lambda]}|q^{-1}\hat{q}_{n}(z)-z|\leq y_{n}\varepsilon\quad\forall n\geq n_{\delta,\varepsilon},\label{eq:qG_QinvQhat-Id}
\end{equation}
which is true if $\hat{\nu}_{n}$ defined by (\ref{eq:nu_def}) satisfies
\begin{equation}
\sup_{z\in[y_{n},y_{n}\Lambda]}|\hat{\nu}_{n}(z)|\rightarrow0\label{eq:nu_conv}
\end{equation}
and also
\begin{equation}
\sup_{t\in[\textrm{e}^{-y_{n}},1]}\bigl|\log(t/\varGamma_{n}^{-1}(t))\bigr|/y_{n}\rightarrow0,\label{eq:ynu_conv}
\end{equation}
with $\varGamma_{n}^{-1}$ the empirical quantile function of $\mathcal{U}^{(1)},..,\mathcal{U}^{(n)}$
as in the proof of Lemma \ref{lem: for Bn} (note that $q^{-1}\hat{q}_{n}(z)=-\log\varGamma_{n}^{-1}(\textrm{e}^{-z})$
for all $z\in[0,y_{n}]$). As in the proof of Lemma \ref{lem: for Bn},
(\ref{eq:SWupp}) and (\ref{eq:SWlow}) hold. Therefore, since the
upper bound in (\ref{eq:k0_for_estimator}) implies that $\liminf_{n\rightarrow\infty}y_{n}/\log n>0$,
(\ref{eq:ynu_conv}) holds almost surely. Moreover, by (\ref{eq:nu_conv-0}),
(\ref{eq:nu_conv}) holds almost surely. This proves the univariate
case.\end{svmultproof}

\begin{lemma}
\label{lem: sampling_LDL}Let the random vector X on $\mathbb{R}^{m}$
have continuous marginals satisfying log-GW tail limits and let $Y$
defined by (\ref{eq:Ydef}) satisfy the LDP (\ref{eq:LDP_mexp}) with
good rate function $I$. Let $X^{(1)},X^{(2)},...$ be a sequence
of iid copies of $X$. Let $(k_{n})$ satisfy (\ref{eq:k0_for_estimator})
and let the quantile estimator $\hat{q}_{j,n}$ given by (\ref{eq:q_hat})
satisfy (\ref{eq:nu_conv-0}). Let the Borel set $A\subset[0,\infty)^{m}$
be a continuity set of $I$ satisfying $\inf I(A)\in(0,\infty)$.
Then $\hat{\mu}$ defined by (\ref{eq:mu_hat_def}) satisfies for
every $\Lambda>1$ and every
\begin{equation}
\Delta\in\Bigl(0,\frac{1}{(1-c')\inf I(A)}\Bigr)\label{eq:Delta_condition}
\end{equation}
and $\delta\in(0,\Delta)$: 
\begin{equation}
\lim_{n\rightarrow\infty}\sup_{\lambda\in[\Lambda^{-1},\Lambda],\: d\in[\delta,\Delta]}\bigl|y_{n}^{-1}\log\hat{\mu}_{n,d/\lambda}(A\lambda)+d\inf I(A)\bigr|=0\quad a.s.\label{eq:Claim_random}
\end{equation}
and\textup{
\begin{equation}
\limsup_{n\rightarrow\infty}\sup_{\lambda\in[\Lambda^{-1},\Lambda],\: d>\Delta}y_{n}^{-1}\log\hat{\mu}_{n,d/\lambda}(A\lambda)\leq-\Delta\inf I(A)\quad a.s.\label{eq:Claim_random_ubound}
\end{equation}
}\end{lemma}
\begin{svmultproof}
With $\{prop\}$ denoting the subset of $\Omega$ satisfying the proposition
$prop$, consider
\begin{equation}
\begin{array}{c}
C_{a,b,n}^{(i)}:=\{\sup_{l\geq a}\bigl\Vert Q^{-1}\hat{Q}_{n}(\hat{Y}_{n}^{(i)}l^{-1})-Y^{(i)}l^{-1}\bigr\Vert_{\infty}>y_{n}b\}.\end{array}\label{eq:def_C}
\end{equation}
for $i=1,..,n$, which are elements of $\mathcal{F}_{n}$. Following
(\ref{eq:p_hat_def}), we can define empirical probabilities $\hat{p}_{n}(C_{a,b,n}):=n^{-1}\Sigma_{i\in\{1,..,n\}}\mathbf{1}(C_{a,b,n}^{(i)})$.
Combining Lemmas \ref{lem: for Bn} and \ref{lem: For Cn} gives
\begin{equation}
\lim_{n\rightarrow\infty}y_{n}^{-1}\log\hat{p}_{n}(C_{a,b,n})=-\infty\quad a.s.\quad\forall a,b>0.\label{eq:C_bnd}
\end{equation}

For every $S\subset\mathbb{R}^{m}$ and $\iota>0$, let $S^{\iota}:=\{x\in\mathbb{R}^{m}:\:\inf_{x'\in S}\bigl\Vert x-x'\bigr\Vert_{\infty}\leq\iota\}$
(closed), and $S^{-\iota}:=\{x\in\mathbb{R}^{m}:\:\inf_{x'\in S^{c}}\bigl\Vert x-x'\bigr\Vert_{\infty}>\iota\}$
(open). Set $S^{0}:=S$. Since $I$ is a good rate function, Lemma
4.1.6 of \citet{Dembo=000026Zeitouni} implies 
\begin{equation}
\lim_{\iota\downarrow0}\inf I(A{}^{\iota})=\inf I(\bar{A})=\inf I(A^{o})=\inf I(\cup_{\iota>0}A{}^{-\iota})=\lim_{\iota\downarrow0}\inf I(A{}^{-\iota}),\label{eq:lbl}
\end{equation}
so the nonincreasing function $\iota\mapsto\inf I(A^{\iota})$ is
continuous in $(-\iota_{0},\iota_{0})$ for some $\iota_{0}>0$, and
therefore, $A{}^{\iota}$ is a continuity set of $I$ for every $\iota\in(-\iota_{0},\iota_{0})$.
Moreover, by (\ref{eq:Delta_condition}), there exist $\varepsilon>0$
and $\iota_{1}\in(0,\iota_{0})$ such that $\inf I(A{}^{-\iota})\leq\inf I(A)+\varDelta^{-1}\varepsilon<\varDelta^{-1}(1-c')^{-1}$
for all $\iota\in[0,\iota_{1}]$. Therefore, for
\begin{equation}
\begin{array}{c}
E_{d,\iota,n}^{(i)}:=\{Y^{(i)}\in dy_{n}A{}^{\iota}\}\in\mathcal{F}_{n},\quad i=1,..,n,\end{array}\label{eq:def_E}
\end{equation}
Lemma \ref{lem: for En} implies for $\Delta$ satisfying (\ref{eq:Delta_condition})
and every $\delta\in(0,\Delta)$ that
\begin{equation}
\lim_{n\rightarrow\infty}\sup_{d\in[\delta,\Delta]}\bigl|y_{n}^{-1}\log\hat{p}_{n}(E_{d,\iota,n})+d\inf I(A^{\iota})\bigr|=0\quad a.s.\quad\forall\iota\in[-\iota_{1},\iota_{1}].\label{eq:E_lim}
\end{equation}

Therefore, 
\begin{equation}
\liminf_{n\rightarrow\infty}\inf_{d\in[\delta,\Delta]}y_{n}^{-1}\log\hat{p}_{n}(E_{d,-\iota,n})\geq-(1-c')^{-1}\quad\forall\iota\in[0,\iota_{1}]\quad a.s.\label{eq:E_bnd}
\end{equation}

Let 
\begin{equation}
D_{\lambda,d,n}^{(i)}:=\{\hat{Q}_{n}(\hat{Y}_{n}^{(i)}\lambda/d)\in Q(y_{n}A\lambda)\}\in\mathcal{F}_{n},\quad i=1,..,n.\label{eq:def_D}
\end{equation}

By (\ref{eq:def_E}), (\ref{eq:def_C}) and (\ref{eq:def_D}), we
have for all $d\geq\delta$, $\lambda\in[\Lambda^{-1},\Lambda]$ and
$\iota>0$ that $E_{d,-\iota,n}^{(i)}\cap(C_{d/\lambda,\iota\lambda,n}^{(i)})^{c}\subset D_{\lambda,d,n}^{(i)}$
and therefore, $\hat{p}_{n}(D_{\lambda,d,n})\geq\hat{p}_{n}(E_{d,-\iota,n})-\hat{p}_{n}(C_{\delta/\Lambda,\iota/\Lambda,n})$.
Therefore, for all $\iota\in(0,\iota_{1}]$,
\[
\liminf_{n\rightarrow\infty}\inf_{d\in[\delta,\Delta],\:\lambda\in[\Lambda^{-1},\Lambda]}y_{n}^{-1}\log\hat{p}_{n}(D_{\lambda,d,n})-y_{n}^{-1}\log\hat{p}_{n}(E_{d,-\iota,n})
\]
\begin{equation}
\geq\liminf_{n\rightarrow\infty}y_{n}^{-1}\log\left(1-\textrm{e}^{\log\hat{p}_{n}(C_{\delta/\Lambda,\iota/\Lambda,n})-\inf_{d\in[\delta,\Delta]}\log\hat{p}_{n}(E_{d,-\iota,n})}\right)\geq0\quad a.s.,\label{eq:DE_ineq}
\end{equation}
the last inequality following from (\ref{eq:C_bnd}) and (\ref{eq:E_bnd}).
Therefore, by (\ref{eq:E_lim}) and (\ref{eq:lbl}),
\begin{equation}
\liminf_{n\rightarrow\infty}\inf_{d\in[\delta,\Delta],\:\lambda\in[\Lambda^{-1},\Lambda]}y_{n}^{-1}\log\hat{p}_{n}(D_{\lambda,d,n})+d\inf I(A)\geq0\quad a.s.\label{eq:pnZ_lLDP}
\end{equation}

For all $d\geq\delta$, $\lambda\in[\Lambda^{-1},\Lambda]$ and $\iota>0$,
we have $D_{\lambda,d,n}^{(i)}\cap(C_{\delta/\Lambda,\iota/\Lambda,n}^{(i)})^{c}\subset E_{d,\iota,n}^{(i)}$,
so $D_{\lambda,d,n}^{(i)}\subset E_{d,\iota,n}^{(i)}\cup C_{\delta/\Lambda,\iota/\Lambda,n}^{(i)}$
and
\begin{equation}
\hat{p}_{n}(D_{\lambda,d,n})\leq2\max(\hat{p}_{n}(C_{\delta/\Lambda,\iota/\Lambda,n}),\hat{p}_{n}(E_{d,\iota,n})).\label{eq:pDbound}
\end{equation}

Therefore, by (\ref{eq:C_bnd}), (\ref{eq:E_lim}) and (\ref{eq:lbl}),
\begin{equation}
\limsup_{n\rightarrow\infty}\sup_{d\in[\delta,\Delta],\:\lambda\in[\Lambda^{-1},\Lambda]}y_{n}^{-1}\log\hat{p}_{n}(D_{\lambda,d,n})+d\inf I(A)\leq0\quad a.s.,\label{eq:pnZ_uLDP_iota}
\end{equation}
so with (\ref{eq:pnZ_lLDP}) and (\ref{eq:mu_hat_def}), (\ref{eq:Claim_random})
is obtained. By Lemma \ref{lem: for En}, with $\iota_{0}$ as above,
\begin{equation}
\limsup_{n\rightarrow\infty}\sup_{d>\Delta}y_{n}^{-1}\log\hat{p}_{n}(E_{d,\iota,n})\leq-\Delta\inf I(A^{\iota})\quad a.s.\quad\forall\iota\in[0,\iota_{0}).\label{eq:E_bnd_upp}
\end{equation}

Then (\ref{eq:Claim_random_ubound}) follows from (\ref{eq:pDbound})
using (\ref{eq:C_bnd}), (\ref{eq:E_bnd_upp}) and (\ref{eq:lbl}). \end{svmultproof}

\begin{acknowledgements}
The author is grateful to the Associate Editor and two anonymous Referees,
John Einmahl and Laurens de Haan for their comments and suggestions,
which made the manuscript much better. The support of Rijkswaterstaat
by making the oceanographic data available is gratefully acknowledged. 
\end{acknowledgements}

\noindent \begin{flushleft}
\textbf{\small{}Conflict of Interest }{\small{}The author declares
that he has no conflict of interest.}
\par\end{flushleft}{\small \par}

\end{document}